\documentclass[12pt,final]{amsart}

\usepackage{mathrsfs}
\usepackage{txfonts}
\usepackage{amssymb,dsfont}
\usepackage{enumerate}
\usepackage{slashed}
\usepackage{fancyhdr}
\usepackage{aliascnt}
\usepackage[backref,pagebackref,linkcolor=red,citecolor=green,filecolor=magenta,urlcolor=cyan]{hyperref}
\usepackage[square,super,sort,longnamesfirst]{natbib}
\usepackage[dvips,dvipsnames]{xcolor}
\usepackage[dvips]{graphicx}
\usepackage[all,cmtip]{xy}
\usepackage{pdfsync}

\makeatletter
\newcommand{\refcheckize}[1]{%
  \expandafter\let\csname @@\string#1\endcsname#1%
  \expandafter\DeclareRobustCommand\csname relax\string#1\endcsname[1]{%
    \csname @@\string#1\endcsname{##1}\wrtusdrf{##1}}%
  \expandafter\let\expandafter#1\csname relax\string#1\endcsname
}
\makeatother

\newcommand{\norm}[1]{\left\lVert#1\right\rVert}
\newcommand{\abs}[1]{\left\lvert#1\right\rvert}
\newcommand{\set}[1]{\left\{#1\right\}}
\newcommand{\hin}[2]{\left\langle#1,#2\right\rangle}
\newcommand{\rin}[2]{\left(#1,#2\right)}
\newcommand{\field}[1]{\mathbb{#1}}
\newcommand{\R}{\field{R}}
\newcommand{\Com}{\field{C}}
\newcommand{\Lg}[1]{\mathrm{#1}}
\newcommand{\Cl}{\Lg{Cl}}%
\newcommand{\La}[1]{\mathfrak{#1}}
\newcommand{\D}{\slashed{D}}
\newcommand{\dirac}{\slashed{\partial}}
\newcommand{\RE}{\operatorname{Re}}
\newcommand{\IM}{\operatorname{Im}}

\newcommand{\dif}{\mathrm d}
\newcommand{\diff}{\,\mathrm d}

\newcommand{\vect}[1]{\mathbf{#1}}
\newcommand{\To}{\longrightarrow}
\newcommand{\Rmn}[1]{\uppercase\expandafter{\romannueral#1}}

\DeclareMathOperator{\Div}{div}
\DeclareMathOperator{\trace}{tr}

\DeclareMathOperator{\dist}{dist}

\DeclareMathOperator{\End}{End}
\DeclareMathOperator{\Id}{Id}

\theoremstyle{plain}
\newtheorem{theorem}{Theorem}[section]

\newaliascnt{lem}{theorem}
\newtheorem{lem}[lem]{Lemma}
 \aliascntresetthe{lem}

\newaliascnt{cor}{theorem}
\newtheorem{cor}[cor]{Corollary}
 \aliascntresetthe{cor}

\newaliascnt{prop}{theorem}
\newtheorem{prop}[prop]{Proposition}
 \aliascntresetthe{prop}

\theoremstyle{remark}
\newtheorem{rem}{Remark}[section]

\newtheorem*{claim}{Claim}

\theoremstyle{definition}
\newtheorem{defn}{Definition}[section]

\numberwithin{equation}{section}

\begin{document}

\title[Estimates for Dirac equations]{Estimates for solutions of  Dirac equations and an application to a geometric elliptic-parabolic problem}

\author[Chen]{Qun Chen}
\address{School of Mathematics and Statistics \\ WuHan University\\ 430072 Hubei, China}
\email{\href{mailto:qunchen@whu.edu.cn}{qunchen@whu.edu.cn}}

\author[Jost]{J\"urgen Jost}
\address{Max Planck Institute for Mathematics in the Sciences\\ Inselstrasse 22\\ 04103 Leipzig, Germany}
\email{\href{mailto:jost@mis.mpg.de}{jost@mis.mpg.de}}

\author[Sun]{Linlin Sun}
\address{School of Mathematics and Statistics \\ WuHan University\\ 430072 Hubei, China}
\email{\href{mailto:sunll101@whu.edu.cn}{sunll101@whu.edu.cn}}

\author[Zhu]{Miaomiao Zhu}
\address{School of Mathematical Sciences, Shanghai Jiao Tong University\\ 800 Dongchuan Road \\ Shanghai, 200240 \\China}
\email{\href{mailto:mizhu@sjtu.edu.cn}{mizhu@sjtu.edu.cn}}

\thanks{
The research leading
to these results has received funding from the European Research
Council under the European Union's Seventh Framework Programme
(FP7/2007-2013) / ERC grant agreement no. 267087.
Qun Chen is partially supported by  NSFC   of
 China. Linlin Sun is partially supported by  CSC   of
 China. Miaomiao Zhu is supported in part by NSFC of China (No. 11601325). The authors thank the Max Planck Institute for Mathematics in the Sciences for good working conditions when this work was
carried out. Miaomiao Zhu would like to thank Professor Bernd Ammann for discussions about Dirac operators. }
\date{\today}

\maketitle
\begin{abstract}
We develop  estimates for the solutions and derive existence and uniqueness results of various local boundary value problems for Dirac equations that improve all relevant results known in the literature. With these estimates at hand, we derive a general existence, uniqueness and regularity theorem for solutions of Dirac equations with such boundary conditions. We also apply these estimates to a new nonlinear elliptic-parabolic problem, the Dirac-harmonic heat flow on Riemannian spin manifolds. This problem is motivated by the supersymmetric nonlinear $\sigma$-model and combines a harmonic heat flow type equation with a  Dirac equation that depends nonlinearly on the flow.

\vskip12pt

\end{abstract}

\vskip12pt

\par

\vskip0.2cm

\section{Introduction}

The Dirac equation is one of the mathematically most important and fruitful structures from physics.  As the name indicates, it was first introduced by Dirac \cite{Dirac1928quantum}. Dirac's original equation is hyperbolic, but the elliptic version, which this paper is concerned with,  appears naturally in geometry. Both solutions on closed manifolds and on manifolds with boundary have found important applications. In this paper, we shall systematically investigate the boundary value problem and derive results that are sharper and stronger than all relevant results known prior to our work. We shall then provide a new application which depends on our regularity, existence and uniqueness results and which could not have been derived with the results known in the literature.

The mathematical history of  boundary value problems for Dirac equations started with the work of Atiyah, Patodi and Singer.
 In their seminal papers \cite{Atiyah1975spectral,Atiyah1975spectral2,Atiyah1976spectral}, they
 introduced a nonlocal  boundary condition for first order elliptic differential operators and established
an index theorem on compact manifolds with boundary. This constitutes a cornerstone of the theory of first order elliptic boundary value problems.

In recent years, important progress has been achieved on various extensions, generalizations and simplifications of the  Atiyah-Patodi-Singer theory and their applications.  In particular,  in the works of Bismut and Cheeger \cite{Bismut1990families},   Boo{\ss}-Bavnbek and Wojciechowski \cite{Booss1993elliptic},  Br\"uning and Lesch  \cite{Bruening2001boundary,Bruening1999spectral},
Bartnik and Chrusciel \cite{Bartnik2005boundary},  Ballmann, Br\"uning and Carron \cite{Ballmann2008regularity},  B\"ar and Ballmann \cite{Bar2012boundary}, etc.,    regularity theorems, index theorems and Fredholm theorems for such kind of elliptic
 boundary value problems have been established.

Although the index theorems and Fredholm theorems give us information or criteria for the existence of solutions, in many cases (for instance
the proof of the positive energy theorem e.g. \cite{Witten1981new, Parker1982on, Gibbons1983positive, Herzlich1997penrose} and
Dirac-harmonic maps, see below),
for an elliptic boundary problem and  given  boundary data, one needs  more precise results about the existence and uniqueness of solutions,
and usually this is based on appropriate global elliptic estimates for the solutions. This is our motivation for studying
the boundary values problems for Dirac equations.

In this paper, we first consider the existence and uniqueness for Dirac equations under a class of local elliptic boundary value conditions $\mathcal{B}$
 (including chiral boundary conditions, MIT bag boundary conditions and J-boundary conditions, see the definitions in \autoref{sec:dh}, c.f. \cite{Bar2012boundary, Hijazi2002eigenvalue}).
A Dirac bundle $E$ over a Riemannian manifold $M^m$ ($m\geq2$) is a Hermitian metric vector bundle of left Clifford modules over $\Cl(M)$,
such that the multiplication by unit vectors in $TM$ is orthogonal and the covariant derivative is a module derivation.
Let $\nabla_0$ be a smooth Dirac connection on $E$ and consider another Dirac connection of the form $\nabla=\nabla_0+\Gamma$,
that is,  $\Gamma \in \Omega^1(\Lg{Ad}(E))$ commutes with the Clifford multiplication. We shall work with the {\it Dirac connection spaces}
$\mathfrak{D}^{p}(E)$ defined by the norm
$$\norm{\Gamma}_p\coloneqq\norm{\Gamma}_{L^{2p}(M)}+\norm{\dif\Gamma}_{L^p(M)}.$$
In particular, $\Gamma$ need not be smooth.

The  Dirac operator associated with the Dirac connection $\nabla$ is defined by
\begin{equation*}
\D\coloneqq e_i\cdot\nabla_{e_i},
\end{equation*}
where $e_i\cdot$ denotes the Clifford multiplication, and  $\set{e_i}$ is a local orthonormal frame on $M$. Here and in the sequel, we use the usual summation convention.
We will establish the existence and uniqueness of solutions of the following Dirac equations:
\begin{equation}\label{eq:Dirac1}
\begin{cases}
\D\psi=\varphi,&M;\\
\mathcal{B}\psi=\mathcal{B}\psi_0,&\partial M,
\end{cases}
\end{equation}
where $\varphi\in L^p(E),\mathcal{B}\psi_0\in W^{1-1/p,p}(E\vert_{\partial M})$. Here and in the sequel, all of the Sobolev spaces of sections of $E$
are associated with the fixed smooth Dirac connection $\nabla_0$.  Setting
\begin{equation*}
p^*>1,\quad \text{if}\ m=2,\quad p^* \geq (3m-2)/4,\quad\text{if}\ m>2.
\end{equation*}
We have the following

\begin{theorem} \label{thm:dirac}Let $E$ be a Dirac bundle over a compact $m$-dimensional ($m\geq2$) Riemannian  manifold $M^m$ with boundary.
Suppose that $\Gamma\in\mathfrak{D}^{p^*}(E)$, then for any $1<p<p^*$, \eqref{eq:Dirac1} admits a unique solution
$\psi\in W^{1,p}(E)$. Moreover, $\psi$ satisfies the following estimate
\begin{equation} \label{main estimate}
\norm{\psi}_{W^{1,p}(E)}\leq c\left(\norm{\varphi}_{L^p(E)}+\norm{\mathcal{B}\psi_0}_{W^{1-1/p,p}(E\vert_{\partial M})}\right),
\end{equation}
where $c=c\left(p,\norm{\Gamma}_{p^*}\right)>0$.
\end{theorem}

This estimate is optimal in dimension 2 in the sense that the exponents cannot be improved. It also improves the known estimates in higher dimensions. (The condition $p^* \geq (3m-2)/4$ for $m>2$  arises from the unique continuation of Jerison \cite{Jerison1986carleman} that we shall need in the proof.) For instance, in the fundamental work of Bartnik and Chru\'sciel\cite{Bartnik2005boundary}, only $L^2$-estimates were developed. While that was sufficient for their Fredholm theory of the Dirac operator, for the nonlinear setting that we shall treat later in this paper, the finer $L^p$-estimates that we obtain here are necessary.
In  \cite{Bartnik2005boundary},  when applying the Fredholm criteria to get the existence of solutions, one needs additional conditions (the mean curvature) on the boundary. Essentially, these conditions imply the triviality of the kernel of the elliptic operators. In our case, this extra condition is unnecessary.

One key observation in our proof of the above theorem is that for a harmonic spinor $\psi\in W^{1,p}(E)$, the homogeneous boundary condition
$\mathcal{B}\psi|_{\partial M}=0$ is equivalent to the zero Dirichlet condition $\psi|_{\partial M}=0$ (see \autoref{prop:chirality} and Remark \ref{rem:3.4}),
which is not the case for general spinors. Thus, the uniqueness problem for \eqref{eq:Dirac1} can be reduced to the triviality of harmonic spinor
with zero Dirichlet boundary value for Dirac operators with a non-smooth connection $\nabla_0+\Gamma$, $\Gamma\in\mathfrak{D}^{p^*}$.
To derive this uniqueness, in dimension $m=2$, inspired by the approach of H\"ormander \cite{Hoermander1965l2estimate}, we establish an $L^2$ estimate with some
suitable weight for our Dirac operators (see \autoref{thm:L2-estimate}); in dimension $m>2$, we apply the weak unique continuation property
 (WUCP) of Dirac type operators $D+V$, where $D$ is a Dirac operator with a smooth connection and $V$ is a potential
 (see \cite{Carleman1939sur} for $V$ continuous, \cite{Booss2002weak} for $V$ bounded, and  \cite{Jerison1986carleman} for $V \in L^{(3m-2)/2}$) and use
some extension argument for Dirac operators on manifolds with boundary as in \cite{Booss1993elliptic}. Finally, by using the uniqueness result of
our boundary value problem $(\D, \mathcal{B})$, we can improve the standard elliptic boundary estimate
for Dirac operators to our main $L^p$ estimate \eqref{main estimate}
(see  \autoref{thm:main-estimate} and Remark \ref{rem: norm-gamma}), which is uniform  in the sense that the constant
$c=c\left(p,\norm{\Gamma}_{p^*}\right) >0$  depends on the $\norm{\cdot}_{p^*}$ norm of $\Gamma$ but not on $\Gamma$ itself -
a property playing an important role in our later application to some elliptic-parabolic problem.

\vskip12pt

We can also  apply  the above result to derive the existence and uniqueness for boundary value problems for Dirac operators along a map. This will also be needed  for the Dirac-harmonic map heat flow introduced below.
Let $M$ be a compact Riemannian spin manifold with boundary $\partial M$, $N$ be a compact Riemannian manifold and $\Phi$  a smooth map from $M$ to $N$. Given a fixed spin structure on $M$, let $\Sigma M$ be the spin bundle of $M$.
On the twisted bundle
 $\Sigma M\otimes\Phi^{-1}TN$ ,  one can  define the {\it Dirac operator $\D$ along the map $\Phi$} \cite{Chen2006dirac},   i.e.,
\begin{equation*}
\D\Psi\coloneqq\dirac\psi^{\alpha}\otimes\theta_{\alpha}+e_i\cdot\psi^{\alpha}\otimes\nabla^{TN}_{\Phi_*(e_i)}\theta_{\alpha}.
\end{equation*}
Here $\Psi=\psi^{\alpha}\otimes\theta_{\alpha}$, $\set{\theta_{\alpha}}$ are local cross-sections of $\Phi^{-1}TN$,
$\set{e_i}$ is a local orthonormal frame of $TM$, $\dirac=e_i\cdot\nabla_{e_i}$ is the usual Dirac operator on the spin bundle over $M$ and $X\cdot$ stands for the Clifford multiplication by the vector field $X$ on $M$.   We say that $\Psi$ is a {\it harmonic spinor along the map $\Phi$ } if $\D\Psi=0$. The chiral boundary value problem for Dirac operators along a map was firstly considered in \cite{Chen2013boundary}, extending the classical chiral boundary value problem for usual Dirac operators firstly introduced in \cite{Gibbons1983positive}.

\begin{theorem}\label{thm:main-dirac}Let $M^m$ ($m\geq2$) be a compact Riemannian spin manifold with boundary $\partial M$, $N$ be a compact Riemannian manifold.
Let $\Phi\in W^{1,2p^*}(M;N)$. Then for every $1<p<p^*$, $\eta\in L^p\left(M;\Sigma M\otimes\Phi^{-1}TN\right)$ and $\mathcal{B}\psi\in W^{1-1/p,p}(\partial M;\Sigma M\otimes\Phi^{-1}TN)$, the following boundary value problem for the Dirac equation
\begin{equation*}
\begin{cases}
\D\Psi=\eta,&M;\\
\mathcal{B}\Psi=\mathcal{B}\psi,&\partial M
\end{cases}
\end{equation*}
admits a unique solution $\Psi\in W^{1,p}(M;\Sigma M\otimes\Phi^{-1}TN)$, where $\D$ is the Dirac operator along the map $\Phi$. Moreover, there exists a constant $c=c\left(p,\norm{\Phi}_{W^{1,{2p}^*}(M)}\right)>0$ such that
\begin{equation*}
\norm{\Psi}_{W^{1,p}(M)}\leq c\left(\norm{\eta}_{L^p(M)}+\norm{\mathcal{B}\psi}_{W^{1-1/p,p}(\partial M)}\right).
\end{equation*}
\end{theorem}

\vskip12pt

We shall then apply these estimates and existence results to a new elliptic-parabolic problem in geometry that involves Dirac equations. The novelty of this problem consists in the combination of a second order semilinear parabolic equation  with a first order elliptic side condition of Dirac type. We see this as a model problem for a heat flow approach to various other first order elliptic problems in geometric analysis. In any case, this is a nonlinear system coupling a Dirac equation with another prototype of a geometric variational problem, that of harmonic maps. The problem is also motivated by the supersymmetric nonlinear $\sigma$-model of QFT, see e.g. \cite{Deligne1999quantum,Jost2009geometry} where the fermionic part is a Dirac spinor. In fact, this is one of the most prominent roles that Dirac equations play in contemporary theoretical physics, as this leads to the action functional of super string theory.

In order to set up that problem, we first  have to recall the notion of {\it Dirac-harmonic maps}.
Consider the following functional
\begin{equation*}
  L(\Phi,\Psi)=\dfrac12\int_M\left(\norm{\dif\Phi}^2+\rin{\Psi}{\D\Psi}\right),
\end{equation*}
where $\rin{}{}=\RE\hin{}{}$ is the real part of the induced Hermitian inner product $\hin{}{}$ on $\Sigma M\otimes\Phi^{-1}TN$.

A Dirac-harmonic map (see \cite{Chen2005regularity,Chen2006dirac}) then is defined to be a critical point $(\Phi,\Psi) $ of $L$. The Euler-Lagrange equations are
\begin{eqnarray}
\tau(\Phi)&=&\dfrac12\rin{\psi^{\alpha}}{e_i\cdot\psi^{\beta}}R^N(\theta_{\alpha},\theta_{\beta})\Phi_*(e_i)\eqqcolon\mathcal{R}(\Phi,\Psi),\label{eq:EL1}\\
\D\Psi&=&0,\label{eq:EL2}
\end{eqnarray}
where $R^N(X,Y)\coloneqq[\nabla^N_X,\nabla^N_Y]-\nabla^N_{[X,Y]}, \forall X,Y\in\Gamma(TN)$ stands for the curvature operator of $N$ and  $\tau(\Phi)\coloneqq(\nabla_{e_i}\dif\Phi)(e_i)$ is the tension field of $\Phi$.

The general regularity and existence problems for Dirac-harmonic maps have been considered in \cite{Chen2005regularity,Chen2006dirac,Wang2009regularity,Zhu2009regularity,Chen2013boundary,Chen2013maximum,Sharp2013regularity}.   The existence of uncoupled Dirac-harmonic maps (in the sense that the map part is harmonic) via the index theory method was obtained in \cite{Ammann2013dirac}.
For the construction of  examples of  coupled Dirac-harmonic maps (in the sense that the map part is not harmonic), we refer to \cite{Jost2009explicit,Ammann2011examples}.

\vskip12pt

Here, we want to propose and develop   an alternative approach to the existence of Dirac-harmonic maps.
This will be the parabolic or heat flow approach. \eqref{eq:EL1} is a second-order elliptic system, and so,
 we can turn it into a parabolic one by letting the solution depend on time $t$ and putting a time derivative on the left hand side.
In contrast, \eqref{eq:EL2} is first order, and so, we cannot convert it into a parabolic equation, but need to carry it as a constraint along the flow.
 Thus, we  introduce the following flow for Dirac-harmonic maps:  For $\Phi\in C^{2,1,\alpha}(M\times(0,T];N)$ and
$\Psi\in C^{1,0,\alpha}(M\times[0,T];\Sigma M\otimes\Phi^{-1}TN)$
\begin{equation}\label{eq:dhf}
\begin{cases}
\partial_t\Phi=\tau(\Phi)-\mathcal{R}(\Phi,\Psi),&M\times(0,T];\\
\D\Psi=0,&M\times[0,T],
\end{cases}
\end{equation}
with the boundary-initial data
\begin{equation}\label{eq:bdhf}
\begin{cases}
\Phi=\phi,&M\times\set{0}\cup\partial M\times[0,T];\\
\mathcal{B}\Psi=\mathcal{B}\psi,&\partial M\times[0,T],
\end{cases}
\end{equation}
where $\phi\in C^{2,1,\alpha}(M\times\set{0}\cup\partial M\times[0,T];N)$ and $\psi\in C^{1,0,\alpha}(\partial M\times[0,T];\Sigma M\times\phi^{-1}TN)$,
and $f\in C^{k,l,\alpha}$ means that $f(x,\cdot)\in C^{l+\alpha/2}$ and $f(\cdot,t)\in C^{k+\alpha}$.
We call this system \eqref{eq:dhf} {\it the heat flow for Dirac-harmonic map.}

We consider this problem as a model for a parabolic approach to other problems in geometric analysis that involve first order  side conditions. Also, we shall use this problem to demonstrate the power of our estimates for Dirac equations.

\vskip12pt

We shall apply our elliptic estimates for Dirac equations with boundary conditions to obtain the local existence and uniqueness of
the heat flow for Dirac-harmonic map. The long time existence will be considered elsewhere, as it involves problems of a different nature. For the classical theory of harmonic map heat flow, we refer to e.g. \cite{Eells1964harmonic,Hamilton1975harmonic, Li1991heat, Chang1989heat,Struwe1985on,Struwe1988on, Lin2008analysis} etc.

\begin{theorem}\label{thm:main}Let $M^m$ ($m\geq2$) be a compact Riemannian spin manifold with boundary $\partial M$, $N$ be a compact Riemannian manifold.
Suppose that
\begin{gather*}
\phi\in\cap_{T>0}C^{2,1,\alpha}(\bar M\times[0,T];N),
\intertext{and}
\mathcal{B}\psi\in \cap_{T>0}C^{1,0,\alpha}(\partial M\times[0,T];\Sigma M\otimes\phi^{-1}TN)
\end{gather*}
 for some $0<\alpha<1$, then the problem consisting of \eqref{eq:dhf} and \eqref{eq:bdhf} admits a unique solution
 \begin{gather*}
 \Phi\in\cap_{0<t<\tau<T_1} C^{2,1,\alpha}(\bar M\times[t,\tau])\cap C^0(\bar M\times[0,T_1];N),
 \intertext{and}
  \Psi\in\cap_{0<t<\tau<T_1}C^{1,0,\alpha}(\bar M\times[t,\tau])\cap C^{2,0,\alpha}(M\times(0,T_1))\cap C^{1,0,0}(\bar M\times[0,T_1];\Sigma M\otimes\Phi^{-1}TN)
 \end{gather*}
 for some time $T_1>0$. The maximum time $T_1$ is characterized by the condition
\begin{equation*}
\limsup_{t<T_1,t\to T_1}\norm{\dif\Phi(\cdot,t)}_{C^0(\bar M)}=\infty.
\end{equation*}
\end{theorem}
\begin{rem}
All our results \autoref{thm:dirac}, \autoref{thm:main-dirac} and \autoref{thm:main} hold for
the chiral boundary operators  $\mathcal{B}^{\pm}\coloneqq\tfrac12\left(\Id\pm\vect{n}\cdot G\right),$   the MIT bag boundary operators
$\mathcal{B}^{\pm}_{MIT}\coloneqq\tfrac12\left(1\pm\sqrt{-1}\vect{n}\right)$, and the $J$-boundary operators
 $\mathcal{B}_J^{\pm}\coloneqq\tfrac12\left(\Id\pm\vect{n}\cdot J\right)$. We will only give the proofs for the case of the chiral boundary conditions.
The proofs for the other cases are similar and hence we will omit them.
\end{rem}

We would like to mention that Branding \cite{Branding2014evolution} considered regularized Dirac-harmonic maps from closed Riemannian surfaces and
studied the corresponding evolution problem.

\vskip12pt

The paper is organized as follows. In \autoref{sec:dh}, we provide the definitions of Dirac bundle etc. and Dirac-harmonic maps.
 We also derive the Euler-Lagrange equation for Dirac-harmonic maps. In \autoref{sec:dirac}, we derive some elliptic estimates and the existence and
uniqueness of solutions of Dirac equations with chiral boundary value conditions. In \autoref{sec:dirac-map}, we will prove \autoref{thm:main-dirac}.
In \autoref{sec:dhf} we give a proof of the short time existence for the flow of Dirac-harmonic maps, \autoref{thm:main}. Finally, in \autoref{sec:dirac-disk}
we discuss a special case of the Dirac equation along a map between Riemannian disks. In this special case, the solution can be given through Cauchy integrals.
\par
\vspace{1ex}
\textbf{Notations}:
\vspace{1ex}
\par
 The lower case letter $c$ will designate a generic constant possibly depending on $M,N$ and other parameters, but independent of a particular solution of \eqref{eq:DHF} and \eqref{eq:BDHF}, while the capital letter $C$ will designate a constant possibly depending on the solutions.
 \par
  We list some notations in the following:
 \begin{list}{}{}
 \item $\Sigma M$ \quad the spin bundle on $M$.
\item $\hom{E}$\quad the homomorphism bundle of $E$.
\item $C^k(M;N)$\quad the space of all $C^k$-maps from $M$ to $N$.
\item $C^k(E)=C^k(M;E)$\quad the space of all $C^k$-cross sections of $E$ where $E$ is a vector bundle on $M$.
\item $C^{k}(\partial M;E)$\quad the space of all $C^k$-cross sections of $E$ restricted to the boundary $\partial M$.
\item $C^{k,l,\alpha}(M\times I;E)$\quad the space of all cross sections $\psi(\cdot,t)$ of $E$ such that $\psi\in C^{k,l,\alpha}(M\times I)$.
\item $W^{s,p}(E)=W^{s,p}(M;E)$.
\item $\End{E}$\quad the endomorphism bundle of $E$.
\item $\Omega^p(E)=\Gamma\left(\Lambda^pT^*M\otimes E\right)$\quad the space of all $E$-valued $p$-forms on $M$.
\item $\Omega^p(\La{so}_n)=\Omega^p(M)\otimes\La{so}_n$.
\item $\Lg{Ad}(E)$\quad a sub-bundle of $\End(E)$ such that for all $A\in\Lg{Ad}(E)$, we have that $A=-A^*$.
\item $A(D)$\quad the space of all holomorphic functions on $D$.
\item $\norm{\cdot}$\quad the inner norm, i.e., $\norm{\psi}^2=\hin{\psi}{\psi}$. We also use the same notation for some special norms in the sequel, as will be specified in the appropriate places.
 \end{list}
 \par
\vspace{3ex}
\par
\par

\vskip0.2cm

\section{Preliminaries}\label{sec:dh}

\subsection{Dirac bundles}
\begin{defn}[See \cite{Lawson1989spin}.]Let $E$ be a Hermitian bundle on a Riemannian manifold $M^m$ of left Clifford modules over $\Cl(M)$. Denote the Clifford multiplication, the metric and the connection by $\cdot,\hin{}{},\nabla$ respectively. We say that $E$ is a Dirac bundle, if the following properties hold:
\begin{enumerate}[D1]
\item The Clifford multiplication is parallel, i.e., the covariant derivative on $E$ is a module derivation, i.e.,
\begin{equation*}
\nabla_X\left(Y\cdot\psi\right)=\nabla_XY\cdot\psi+Y\cdot\nabla_X\psi,\quad\forall X,Y\in\Gamma(TM),\ \psi\in\Gamma(E).
\end{equation*}
\item The Clifford multiplication by unit vectors in $TM$ is orthogonal, i.e.,
\begin{equation*}
\hin{X\cdot\psi}{\varphi}=-\hin{\psi}{X\cdot\varphi},\quad\forall X\in TM,\psi,\varphi\in E.
\end{equation*}
\item The connection is a metric connection, i.e.,
\begin{equation*}
X\hin{\psi}{\varphi}=\hin{\nabla_X\psi}{\varphi}+\hin{\psi}{\nabla_X\varphi},\quad\forall X\in TM,\psi,\varphi\in\Gamma(E).
\end{equation*}
\end{enumerate}
We call such a  connection a {\it Dirac connection}.
\end{defn}
Then one can define the Dirac operator associated to a Dirac bundle by
\begin{equation*}
\D\coloneqq\gamma^E\circ\nabla,
\end{equation*}
where $\gamma^E$ stands for the Clifford multiplication on $E$. In local coordinates, $\D$ is given by
\begin{equation*}
\D=\gamma^E(e_i)\nabla_{e_i}=e^i\cdot\nabla_{e_i}
\end{equation*}
where $\set{e_i}$ is a local orthogonal frame of $TM$. One can check that $\D$ is self-adjoint, i.e., we have the following Green formula
\begin{equation*}
\int_M\hin{\D\psi}{\varphi}=\int_{M}\hin{\psi}{\D\varphi}+\int_{\partial M}\hin{\vect{n}\cdot\psi}{\varphi},\quad\forall\psi,\varphi\in\Gamma(E).
\end{equation*}
\par
 Suppose  $E$ is a Dirac bundle on $M$ and $F=E\vert_{\partial M}$ is the restriction of $E$ to the boundary $\partial M$. Then $F$ is a Dirac bundle in a natural way as follows.
\begin{enumerate}[F1]
\item The metric on $F$ is just the restriction of $E$ on $\partial M$.
\item The Clifford multiplication of $F$, denoted by $\gamma$, is defined as
\begin{equation*}
\gamma(X)\psi\coloneqq\vect{n}\cdot X\cdot\psi,\quad\forall X\in T\partial M,\ \psi\in F.
\end{equation*}
\item The connection $\bar\nabla$ of $F$ is defined as
\begin{equation*}
\bar\nabla_X\psi\coloneqq
\nabla_X\psi+\dfrac{1}{2} \gamma(A(X))\psi,\quad\forall X\in\Gamma(T\partial M),\ \psi\in\Gamma(F),
\end{equation*}
where $A$ is the shape operator of $\partial M$ with respect to the unit outward normal field $\vect{n}$ along $\partial M$.
\end{enumerate}
\begin{lem}This construction gives a Dirac bundle $F$ on $\partial M$.
\end{lem}
\begin{proof}
\begin{enumerate}
\item It is obvious that $\gamma\circ A\in \Omega^1(\Lg{Ad}(F))$. As a consequence, $\bar\nabla$ is a metric connection on $F$. Moreover, $\gamma(X)\in\Gamma(\Lg{Ad}(F))$.
\item Let $B$ be the second fundamental form of $\partial M$ in $M$. For every $X,Y\in\Gamma(T\partial M)$ with $\bar\nabla_XY=0$ at the considered point,
\begin{equation*}
\begin{split}
\bar\nabla_X(\gamma(Y)\psi)=&\nabla_X(\vect{n}\cdot Y\cdot\psi)+\dfrac12A(X)\cdot Y\cdot\psi\\
=&-A(X)\cdot Y\cdot\psi+\vect{n}\cdot B(X,Y)\cdot\psi+\vect{n}\cdot Y\cdot\nabla_X\psi+\dfrac12 A(X)\cdot Y\cdot\psi\\
=&-A(X)\cdot Y\cdot\psi-\hin{A(X)}{Y}\psi+\vect{n}\cdot Y\cdot\nabla_X\psi+\dfrac12 A(X)\cdot Y\cdot\psi\\
=&-A(X)\cdot Y\cdot\psi+\dfrac12 A(X)\cdot Y\cdot\psi+\dfrac12 Y\cdot A(X)\cdot\psi+\vect{n}\cdot Y\cdot\nabla_X\psi+\dfrac12A(X)\cdot Y\cdot\psi\\
=&\vect{n}\cdot Y\cdot\nabla_X\psi+\dfrac 12Y\cdot A(X)\cdot\psi=\gamma(Y)\bar\nabla_X\psi.
\end{split}
\end{equation*}
This identity means that $\gamma$ is parallel.
\end{enumerate}
Therefore, $F$ is a Dirac bundle on $\partial M$.
\end{proof}
\par
The Dirac operator $\bar\D$ of $F$, defined by
\begin{equation*}
\bar\D\coloneqq
\gamma(e_i)\bar\nabla_{e_i},
\end{equation*}
where $\set{e_i}$ is a local orthogonal frame of $T\partial M$, according to the definition, satisfies the following relationship
\begin{equation*}
\bar\D=\vect{n}\cdot\D+\nabla_{\vect{n}}-\dfrac{m-1}{2}h \end{equation*}
where $h$ is the mean curvature of $\partial M$ with respect to $\vect{n}$. If $M$ is a surface,  $-h$ is just the geodesic curvature of $\partial M$ (as a curve) in $M$.
\par
\subsection{Chiral and MIT bag boundary value conditions }In this subsection, we introduce the chiral and MIT bag boundary conditions
(c.f.  \cite{Bar2012boundary, Hijazi2002eigenvalue}). We say that $G$ is a chiral operator if $G\in\Gamma(\End(E))$ satisfies
\begin{equation*}
G^2=\Id,\quad G^*=G,\quad \nabla G=0,\quad GX\cdot=-X\cdot G
\end{equation*}
for every $X\in TM$. It is easy to check that
\begin{equation*}
\gamma(X)G=G\gamma(X),\quad \bar\nabla G=0,\quad \bar\D G=G\bar\D,\quad\bar\D\vect{n}\cdot=-\vect{n}\cdot\bar\D
\end{equation*}
hold on the boundary $\partial M$ for all $X\in T\partial M$. The {\it chiral boundary operator} $\mathcal{B}^{\pm}$ is defined by
\begin{equation*}
\mathcal{B}^{\pm}\coloneqq\dfrac12\left(\Id\pm\vect{n}\cdot G\right).
\end{equation*}
It is obvious that $\left(\mathcal{B}^{\pm}\right)^{*}=\mathcal{B}^{\mp}$ and $\bar\D\mathcal{B}^{\pm}=\mathcal{B}^{\mp}\bar\D$. The chiral boundary operator is elliptic \cite{Bar2012boundary} since
\begin{equation*}
\vect{n}\cdot X\cdot\mathcal{B}^{\pm}=\mathcal{B}^{\mp}\cdot\vect{n}\cdot X,\quad\forall X\in T\partial M.
\end{equation*}
\par
{\it The MIT bag boundary operator} is defined by
\begin{equation*}
\mathcal{B}_{MIT}^{\pm}\coloneqq\dfrac12\left(\Id\pm\sqrt{-1}\vect{n}\right).
\end{equation*}
More generally, when $J\in\Gamma(\End(E))$ satisfies
\begin{equation*}
J^2=-\Id,\quad J^*=-J,\quad\nabla J=0,\quad JX\cdot=X\cdot J
\end{equation*}
for every $X\in TM$,  we can define a boundary operator, called {\it the $J$-boundary operator}, by
\begin{equation*}
\mathcal{B}_J^{\pm}\coloneqq\dfrac12\left(\Id\pm\vect{n}\cdot J\right).
\end{equation*}
If $J=\sqrt{-1}$, then it is easy to see that the MIT bag boundary operator is just the $\sqrt{-1}$-boundary operator. Another example is  $J=\sqrt{-1}G_1G_2$ where $[G_1,G_2]=0$ with $G_1,G_2$ being chiral operators. In fact, in our setting, the $J$-operator is just a multiplication of the chiral operator $G$ by $e_1\cdot e_2\cdot$ and vise versa. One can check that $\mathcal{B}_J^{\pm}$ is elliptic since
\begin{equation*}
\mathcal{B}_J^{\pm}\cdot\vect{n}\cdot X=\vect{n}\cdot X\cdot\mathcal{B}_J^{\mp},\quad\forall X\in T\partial M.
\end{equation*}
For simplicity,  we shall denote by $\mathcal{B}$ one of $\mathcal{B}^{\pm}$, $\mathcal{B}_{MIT}^{\pm}$ and $\mathcal{B}_{J}^{\pm}$.
For the sake of convenience, in the sequel, we will mainly consider the case of chiral boundary conditions and omit the detailed discussions of the other cases of
boundary conditions.

\par
The following theorem is well known, see \cite{Bar2012boundary,Chen2013maximum,Sharp2013regularity,Schwarz1995hodge}.
\begin{theorem}[See \cite{Schwarz1995hodge}, p.55, Theorem 1.6.2]\label{thm:fredholm} The operator
\begin{equation*}
(\D,\mathcal{B}):W^{s,p}(E)\To W^{s-1,p}(E)\times W^{s-1/p,p}(E\vert_{\partial M})
\end{equation*}
is Fredholm for all $s\geq 1$ and $1<p<\infty$. Moreover the kernel and co-kernel are independent of the choice of $s$ and $p$. Therefore, we have the following elliptic a-priori estimate
\begin{equation*}
\norm{\psi}_{W^{s,p}(E)}\leq c\left(\norm{\D\psi}_{W^{s-1,p}(E)}+\norm{\mathcal{B}\psi}_{W^{s-1/p,p}(E\vert_{\partial M})}+\norm{\psi}_{L^p(E)}\right),
\end{equation*}
where $c=c(p,s,M,\partial M,\D,\mathcal{B})>0$.
\end{theorem}
\begin{proof}It is a consequence of the fact that $(\D,\mathcal{B})$ is an elliptic operator for $s\geq1$ and $1<p<\infty$.
\end{proof}
\par
\subsection{Dirac connection spaces}
Let $E$ be a Dirac bundle. We consider the affine space of those connection $\nabla$ for which $E$ is again a Dirac bundle. Choose a connection $\nabla_0$, then for any other connection $\nabla$ on $E$,
\begin{equation*}
\nabla=\nabla_0+\Gamma.
\end{equation*}
Here $\Gamma\in\Omega^1(\End(E))$.
\begin{lem}Suppose $\nabla_0$ is a Dirac connection, then $\nabla\coloneqq\nabla_0+\Gamma$ is a Dirac connection if and only if
\begin{equation*}
\Gamma\in\Omega^1(\Lg{Ad}(E)),\quad [\Gamma,\gamma^E]=0,
\end{equation*}
where $\gamma^E$ denotes the Clifford multiplication of $E$.
\end{lem}
\begin{proof}We only need to check that $\gamma^E$ is parallel. For every $X,Y\in\Gamma(TM)$ with $\nabla_XY=0$ at the considered point, we have
\begin{equation*}
\nabla_X(\gamma^E(Y)\psi)=\nabla_{0X}(\gamma^{E}(Y)\psi)+\Gamma(X)\gamma^E(Y)\psi=\gamma^E(Y)\nabla_{0X}\psi+\gamma^E(Y)\Gamma(X)\psi=\gamma^E(Y)\nabla_X\psi.
\end{equation*}
\end{proof}
Introduce $\slashed\Gamma\coloneqq\gamma^E\circ\Gamma=\gamma^E(e_i)\Gamma(e_i)$, then
\begin{equation*}
\D=\D_0+\slashed\Gamma,
\end{equation*}
where $\D,\D_0$ are the Dirac operators associated to the connection $\nabla,\nabla_0$ respectively. From now on, we will consider the modified non-smooth connection $\nabla_0+\Gamma$, denoted by $\nabla$. All of the Sobolev spaces are associated with some fixed smooth connection $\nabla_0$. It is well known that this definition of Sobolev spaces is independent of the choice of $\nabla_0$ if $M$ is compact. However, the connection $\nabla$ need not be smooth, i.e., we only assume that $\Gamma$ belongs to some special function space. For example,
\begin{equation*}
\dif\Gamma\in L^{p^*}(M),\quad\Gamma\in L^{2p^*}(M),
\end{equation*}
where
\begin{equation*}
p^*>1,\quad\text{if}\ m=2,\quad p^* \geq \dfrac{3m-2}{4},\quad\text{if}\ m>2.
\end{equation*}
 \begin{defn}For $p\geq1$, define $\mathfrak{D}^p(E)$ to be the completion of  the subspace of $\Omega^1\left(\Lg{Ad}(E)\right)$ defined by
\begin{gather*}
\mathfrak{D}(E)=\set{\Gamma\in\Omega^1\left(\Lg{Ad}(E)\right):[\Gamma,\gamma^E]=0},
\intertext{with respect to the norm}
\norm{\Gamma}_p\coloneqq\norm{\Gamma}_{L^{2p}(M)}+\norm{\dif\Gamma}_{L^p(M)}.
\end{gather*}
We call these spaces the {\it Dirac connection spaces}.
\end{defn}

\subsection{Dirac-harmonic maps}
Let $(M^m,g)$ be a compact Riemannian spin manifold with (possibly empty) boundary $\partial M$, and $(N^n,h)$ be a compact Riemannian manifold. Concerning the definition  and  properties of  Riemannian spin manifolds, we refer the reader to \cite{Lawson1989spin} for more background material. For any $(\Phi,\Psi)\in C^{1}(M,N)\times\Gamma(\Sigma M\otimes\Phi^{-1}TN)$, we consider the following functional \cite{Chen2005regularity}
\begin{equation*}
  L(\Phi,\Psi)=\dfrac12\int_M\left(\norm{\dif\Phi}^2+\rin{\Psi}{\D\Psi}\right),
\end{equation*}
where $\rin{}{}=\RE\hin{}{}$ is the real part of the Hermitian inner product $\hin{}{}$.

A Dirac-harmonic map (see \cite{Chen2005regularity,Chen2006dirac}) then is defined to be a critical point $(\Phi,\Psi) $ of $L$. The Euler-Lagrange equations are
\begin{equation*}
\begin{cases}
\tau(\Phi)=\dfrac12\rin{\psi^{\alpha}}{e_i\cdot\psi^{\beta}}R^N(\theta_{\alpha},\theta_{\beta})\Phi_*(e_i)\eqqcolon\mathcal{R}(\Phi,\Psi),\\
\D\Psi=0,
\end{cases}
\end{equation*}
where $R^N(X,Y)\coloneqq[\nabla^N_X,\nabla^N_Y]-\nabla^N_{[X,Y]},X,Y\in\Gamma(TN)$ stands for the curvature operator of $N$ and  $\tau(\Phi)\coloneqq(\nabla_{e_i}\dif\Phi)(e_i)$ is the tension field of $\Phi$.

Embed $N$ into $\R^q$ isometrically for some  integer $q$. We may assume  there is a bounded tubular neighborhood $\tilde N$ of $N$ in $\R^q$. Let $\pi:\tilde N\To N$ be the nearest point projection. We may assume  $\pi$ can be extended smoothly to the whole $\R^q$ with compact support. Now we can derive the Euler-Lagrange equation for $L$. Let $\Phi:M\To N$ with $\Phi=(\Phi^A)$, and a spinor $\Psi=\Psi^A\otimes\partial_A\circ\Phi$ along the map $\Phi$ with $\Psi=(\Psi^{A})$ where $\Psi^A$ are spinors over $M$,and  $\partial_A=\partial/\partial z^A$. Notice that $\dif\pi\vert_{N}$ is an orthogonal projection and $\dif\pi(T^{\perp}N)=0$. In fact,
 \begin{gather*}
 \dif\pi(X)=X,\quad\forall X\in TN,
 \intertext{and}
 \dif\pi(\xi)=0,\quad\forall \xi\in T^{\perp}N
 \end{gather*}
where $T^{\perp}N$ is the normal bundle of $N$ in $\R^q$.
 Hence, restricted to $N$, we have
 \begin{equation*}
\pi^A_B\pi^B_C=\pi^A_C,\quad \pi^A_B=\pi^B_A.
 \end{equation*} It is easy to check that
\begin{equation*}
\nu^A_B(\Phi)\nabla\Phi^B=0,\quad\nu^A_B(\Phi)\Psi^B=0,
\end{equation*}
where $\nu^A_B\coloneqq\delta^A_B-\pi^A_B$.
\par
For any smooth map $\eta \in C_{0}^{\infty}(M, \R^q)$ and any smooth spinor field $\xi \in C_{0}^{\infty}(\Sigma M \otimes \R^q)$, we consider the variation
\begin{equation*}
 \Phi_t=\pi(\Phi+t\eta),\quad \Psi^A_t=\pi^A_B(\Phi_t)(\Psi^B+t\xi^B).
\end{equation*}
It is easy to check that
\begin{gather*}
 \Phi_0=\Phi,\quad\Psi_0=\Psi
\intertext{and}
\left.\dfrac{\partial\Phi^A_t}{\partial t}\right\vert_{t=0}=\pi^A_B(\Phi)\eta^B,\quad\left.\dfrac{\partial\Psi^A_t}{\partial t}\right\vert_{t=0}=\pi^A_B(\Phi)\xi^B+\pi^A_{BC}(\Phi)\pi^C_D(\Phi)\Psi^B\eta^D,
\end{gather*}
where
\begin{gather*}
\pi^A_B=\dfrac{\partial\pi^A}{\partial z^B},\quad \pi^A_{BC}=\dfrac{\partial^2\pi^A}{\partial z^B\partial z^C},\quad\dotsc.
\end{gather*}
Moreover,
\begin{equation}\label{eq:pi}
\pi^A_{BC}(\Phi)\pi^C_D(\Phi)=\pi^B_{AC}(\Phi)\pi^C_D(\Phi),\quad \pi^A_{BC}=\pi^A_{CB}.
\end{equation}
Then we have
\begin{prop}\label{prop:EL}
 The Euler-Lagrange equations for $L$ are
\begin{gather*}
 \Delta\Phi^A=\pi^A_{BC}(\Phi)\hin{\nabla\Phi^B}{\nabla\Phi^C}+\pi^A_B(\Phi)\pi^C_{BD}(\Phi)\pi^{C}_{EF}(\Phi)\rin{\Psi^D}{\nabla\Phi^E\cdot\Psi^F},
\intertext{and}
\dirac\Psi^A=\pi^A_{BC}(\Phi)\nabla\Phi^B\cdot\Psi^C.
\end{gather*}
\end{prop}
\begin{rem}\label{rem:2.5}
Denote
\begin{gather*}
\Omega^A_B\coloneqq\nu^A_C(\Phi)\dif\nu^C_B(\Phi)-\dif\nu^A_C(\Phi)\nu^C_B(\Phi)=[\nu(\Phi),\dif\nu(\Phi)]^A_B,\\
R^A_{GDF}\coloneqq\pi^A_B\pi^C_{BD}\pi^G_E\pi^C_{EF}-\pi^G_B\pi^C_{BD}\pi^A_E\pi^C_{EF}
\intertext{and}
\tilde\Omega^A_G\coloneqq\dfrac12R^A_{GDF}(\Phi)\rin{\Psi^D}{e_i\cdot\Psi^F}\eta^i.
\end{gather*}
Then $\Omega^A_B=-\Omega^B_A, \quad \tilde\Omega^A_B=-\tilde\Omega^B_A$ and the Euler-Lagrange equations for $L$ can be rewritten as follows
\cite{Chen2013boundary}
\begin{equation*}
\begin{cases}
\Delta\Phi^A=-\hin{\Omega^A_B}{\dif\Phi^B}+\hin{\tilde\Omega^A_B}{\dif\Phi^B},\\
\dirac\Psi^A=-\Omega^A_B\cdot\Psi^B.
\end{cases}
\end{equation*}
Using the Clifford multiplication $\cdot$ for the Dirac bundle $\Omega^{*}(M)$, we can also write the above system as follows:
\begin{equation*}
\begin{cases}
\Delta\Phi^A=\Omega^A_B\cdot\dif\Phi^B+\hin{\tilde\Omega^A_B}{\dif\Phi^B},\\
\dirac\Psi^A=-\Omega^A_B\cdot\Psi^B.
\end{cases}
\end{equation*}
\end{rem}

\begin{proof}[Proof of \autoref{rem:2.5}]   The proof is similar to \cite{Chen2013boundary}. However, we will present a proof here using our notations. Introduce
\begin{equation*}
S^{AC}_D\coloneqq\pi^A_B\pi^C_{BD},
\end{equation*}
then
\begin{equation*}
R^A_{GDF}=S^{AC}_DS^{GC}_F-S^{GC}_DS^{AC}_F
\end{equation*}
satisfies
\begin{equation*}
R^A_{GDF}=-R^G_{ADF}=-R^A_{GFD}.
\end{equation*}
Moreover,
\begin{align*}
\hin{\tilde\Omega^A_B}{\dif\Phi^B}=&\dfrac12R^A_{GDF}(\Phi)\rin{\Psi^D}{\nabla\Phi^G\cdot\Psi^F}\\
=&S^{AC}_DS^{GC}_F\rin{\Psi^D}{\nabla\Phi^E\cdot\Psi^F}\\
=&\pi^A_B\pi^C_{BD}\pi^C_{EF}\rin{\Psi^D}{\nabla\Phi^E\cdot\Psi^F}.
\end{align*}


Now we need only to check that
\begin{equation*}
\Omega^A_B\wedge\dif\Phi^B=0.
\end{equation*}
Using the fact $\nu^A_B\dif\Phi^B=0$, we have that
\begin{equation*}
\Omega^A_B\wedge\dif\Phi^B=\nu^A_C\dif\nu^C_B\wedge\dif\Phi^B=0.
\end{equation*}
\end{proof}
\begin{proof}[Proof of \autoref{prop:EL}] Note that both $\eta$ and $\xi$ have compact support in $\mathring{M}$. Since
\begin{equation*}
L(\Phi,\Psi)=\dfrac12\int_M\left(\norm{\nabla\Phi^A}^2+\rin{\Psi^A}{\dirac\Psi^A}\right),
\end{equation*}
then, by using the relationship \eqref{eq:pi},
\begin{equation*}
 \begin{split}
  &\left.\dfrac{\dif L(\Phi_t,\Psi_t)}{\dif t}\right\vert_{t=0}=\int_M\hin{\nabla\Phi^A}{\nabla(\pi^A_B\eta^B)}+\dfrac12\int_M\rin{\pi^A_B\xi^B+\pi^A_{BC}\pi^C_D\Psi^B\eta^D}{\dirac\Psi^A}\\
  &+\dfrac12\int_M\rin{\Psi^A}{\dirac\left(\pi^A_B\xi^B+\pi^A_{BC}\pi^C_D\Psi^B\eta^D\right)}\\
  =&\int_M\hin{\nabla\Phi^A}{\pi^A_B\nabla\eta^B+\pi^A_{BC}\nabla\Phi^C\eta^B}+\int_M\rin{\pi^A_B\xi^B+\pi^A_{BC}\pi^C_D\Psi^B\eta^D}{\dirac\Psi^A}\\
  &+\dfrac12\int_{\partial M}\rin{\Psi^A}{\vect{n}\cdot\left(\pi^A_B\xi^B+\pi^A_{BC}\pi^C_D\Psi^B\eta^D\right)}\\
  =&-\int_M\left(\Delta\Phi^A-\pi^A_{BC}\hin{\nabla\Phi^B}{\nabla\Phi^C}-\pi^A_B\pi^C_{BD}\pi^{C}_{EF}\rin{\Psi^D}{\nabla\Phi^E\cdot\Psi^F}\right)\eta^A\\
  &+\int_M\pi^A_B\pi^C_{BD}\rin{\Psi^D}{\dirac\Psi^C-\pi^C_{EF}\nabla\Phi^E\cdot\Psi^F}\eta^A+\int_M\rin{\dirac\Psi^A-\pi^A_{BC}\nabla\Phi^B\cdot\Psi^C}{\xi^A}\\
  &+\int_{\partial M}\nabla_{\vect{n}}\Phi^A\eta^A+\dfrac12\rin{\Psi^A}{\vect{n}\cdot\xi^A}\\
  =&-\int_M\left(\Delta\Phi^A-\pi^A_{BC}\hin{\nabla\Phi^B}{\nabla\Phi^C}-\pi^A_B\pi^C_{BD}\pi^{C}_{EF}\rin{\Psi^D}{\nabla\Phi^E\cdot\Psi^F}\right)\eta^A\\
  &+\int_M\pi^A_B\pi^C_{BD}\rin{\Psi^D}{\dirac\Psi^C-\pi^C_{EF}\nabla\Phi^E\cdot\Psi^F}\eta^A+\int_M\rin{\dirac\Psi^A-\pi^A_{BC}\nabla\Phi^B\cdot\Psi^C}{\xi^A},
 \end{split}
\end{equation*}
where $\vect{n}$ is the unit outward normal vector field along $\partial M$.
\end{proof}

\vskip0.2cm

\par

\section{Existence and uniqueness of solutions of Dirac equations with chiral boundary  conditions}\label{sec:dirac}
In this section, we suppose that $M^m$  ($m\geq2$) is a compact Riemannian spin manifold with boundary $\partial M$ and $E$ is a Dirac bundle on $M$.
We want to derive an existence and uniqueness result for solutions of Dirac equations with chiral boundary  conditions.
The key observation is that for harmonic spinors, the homogeneous chiral condition is equivalent to the zero Dirichlet boundary condition. With this observation,
we can derive a useful $L^2$-estimate for solutions of Dirac equations with chiral boundary  conditions.
The assumption that the boundary is non-empty is essential here. Another application of this observation is that
one can derive Schauder boundary estimates.
\par

\par
\subsection{A property of chiral boundary conditions}
Importantly, the chiral boundary condition is conformally invariant. Moreover, it satisfies
\begin{prop}\label{prop:chirality}Suppose that $\nabla$ is a smooth Dirac connection, then for every $\psi\in H^1(E)$, we have
\begin{equation*}
\abs{\int_{\partial M}\left(\norm{\psi}^2-2\norm{\mathcal{B}\psi}^2\right)}\leq2\norm{\psi}_{L^2(E)}\norm{\D\psi}_{L^2(E)}.
\end{equation*}
\end{prop}
\begin{proof}First, we assume that $\psi$ is smooth. Introduce a vector field
\begin{equation*}
X\coloneqq\dfrac12\rin{\psi}{e_i\cdot G\psi}e_i,
\end{equation*}
then
\begin{gather*}
\hin{X}{\vect{n}}=\dfrac12\hin{\psi}{\vect{n}\cdot G\psi},\\
\norm{\mathcal{B^{\pm}}\psi}^2=\dfrac12\norm{\psi}^2\pm\hin{X}{\vect{n}},
\intertext{and}
\Div X=-\rin{\D\psi}{G\psi}.
\end{gather*}
Using these facts and  integrating by parts, we get that
\begin{equation*}
\abs{\int_{\partial M}\left(\norm{\psi}^2-2\norm{\mathcal{B}\psi}^2\right)}=2\abs{\int_M\rin{\D\psi}{G\psi}}\leq2\norm{\D\psi}_{L^2(E)}\norm{\psi}_{L^2(E)}.
\end{equation*}
The general case follows since $\Gamma(E)$ is dense in $H^1(E)$.
\end{proof}
\begin{rem}This proposition says that the following two systems
\begin{equation*}
\begin{cases}
\D\psi=0,&M;\\
\mathcal{B}\psi=0,&\partial M,
\end{cases}
\quad
\begin{cases}
\D\psi=0,&M;\\
\psi=0,&\partial M,
\end{cases}
\end{equation*}
are equivalent. This fact is important for our whole theory of Dirac equations. With this observation, we then can solve Dirac equations with chiral boundary value conditions.
\end{rem}
\par
\begin{rem}
  \autoref{prop:chirality} also holds for $J$-boundary operators. The proof is similar to \autoref{prop:chirality} and we omit it here. Moreover, all  the results in the sequel associated to  chiral boundary values are also valid for $J$-boundary values. Again, we omit the proof.
\end{rem}
\par

 \par
 \subsection{Regularity of weak solutions and elliptic estimates}
\begin{defn}Suppose that $\slashed\Gamma\in L^{p'}_{loc}(M)$. Let $\psi,\varphi\in L^{p}_{loc}(E)$ where $1/p+1/p'=1$ with $p\geq 1$. We call $\varphi$  a weak solution of the Dirac equation $\D\psi=\varphi$ if
\begin{equation*}
\int_M\hin{\varphi}{\eta}=\int_{M}\hin{\psi}{\D\eta},
\end{equation*}
holds for all smooth spinors $\eta\in\Gamma_0(E)$ of $E$ with compact support in the interior of $M$.
\end{defn}
Define
\begin{equation*}
\hat{m}>2\quad\text{for}\ m=2;\quad\hat{m}=m,\quad \text{for}\ m>2
\end{equation*}
\par
First, we have the following regularity results.
\begin{theorem}[Regularity of weak solutions, see \cite{Bartnik2005boundary,Bar2012boundary}]\label{thm:regularity_weak}Let $M,E$ be as in \autoref{thm:dirac}. Suppose that $\slashed\Gamma\in W^{k,\hat{m}}_{loc}(M)$. Let $\psi\in L^2_{loc}(E)$ be a weak solution of $\D\psi=\varphi$ with $\varphi\in H^k_{loc}(E)$, then $\psi\in H^{k+1}_{loc}(E)$.
\end{theorem}
\begin{theorem}[$L^p$-estimate]\label{thm:Lp-estimate}Let $M,E$ be as in \autoref{thm:dirac}. Suppose that $p\in(1,\infty),\slashed\Gamma\in L^{\mu}(M)$ where $\mu=m$ if $p<m$ and $\mu>p$ if $p\geq m$. Let $\psi\in W^{1,p}(E)$  be a solution of $\D\psi=\varphi$ with $\varphi\in L^p(E)$, then there exists a constant $c=c\left(p,\slashed\Gamma\right)>0$ with
\begin{equation*}
\norm{\psi}_{W^{1,p}(E)}\leq c\left(\norm{\varphi}_{L^p(E)}+\norm{\mathcal{B}\psi}_{W^{1-1/p,p}(E\vert_{\partial M})}+\norm{\psi}_{L^p(E)}\right).
\end{equation*}
\end{theorem}
\begin{proof}For every $\varepsilon>0$, decompose $\slashed\Gamma=\slashed\Gamma_{\varepsilon}+\slashed\Gamma_{\infty}$ with
\begin{equation*}
\norm{\slashed\Gamma_{\varepsilon}}_{L^{\mu}(E)}\leq\varepsilon,\quad\norm{\slashed\Gamma_{\infty}}_{L^{\infty}(E)}\leq c(\varepsilon).
\end{equation*}
 Noticing that
\begin{equation*}
\D_0\psi=\varphi-\slashed\Gamma\psi=\varphi-\slashed\Gamma_{\varepsilon}\psi-\slashed\Gamma_{\infty}\psi,
\end{equation*}
we have
\begin{equation*}
\begin{split}
\norm{\D_0\psi}_{L^p(E)}\leq&\norm{\varphi}_{L^p(E)}+\norm{\slashed\Gamma_{\varepsilon}\psi}_{L^p(E)}+\norm{\slashed\Gamma_{\infty}\psi}_{L^{p}(E)}\\
\leq&\norm{\varphi}_{L^p(E)}+\norm{\slashed\Gamma_{\varepsilon}}_{L^{\mu}(E)}\norm{\psi}_{L^{p\mu/(\mu-p)}(E)}+\norm{\slashed\Gamma_{\infty}}_{L^{\infty}(E)}\norm{\psi}_{L^{p}(E)}\\
\leq&\norm{\varphi}_{L^p(E)}+\varepsilon\norm{\psi}_{W^{1,p}(E)}+c(\varepsilon)\norm{\psi}_{L^{p}(E)}
\end{split}
\end{equation*}
since $1/p-1/\mu=1/p-1/m$ for $p<m$ and
\begin{equation*}
1/p-1/\mu>1/p-1/m
\end{equation*}
for $\mu>p\geq m$. Hence, by \autoref{thm:fredholm}, for suitable $\varepsilon>0$, we get that
\begin{equation*}
\norm{\psi}_{W^{1,p}(E)}\leq c\left(\norm{\varphi}_{L^p(E)}+\norm{\mathcal{B}\psi}_{W^{1-1/p,p}(E\vert_{\partial M})}+\norm{\psi}_{L^p(E)}\right).
\end{equation*}
\end{proof}

\begin{rem}\label{rem:mu>2}
If $\mu>m$, we can choose  $c=c\left(p,\norm{\slashed\Gamma}_{L^{\mu}(M)}\right)>0$.
\end{rem}
\vspace{3ex}
\par
{\it Proof of \autoref{rem:mu>2}.}
We need only to check  the case of a smooth spinor, i.e., $\psi\in\Gamma(E)$. If not, suppose that there exists a sequence $\psi_n\in\Gamma(E)$ and $\slashed\Gamma_n\in L^{\mu}(M)$ such that
\begin{gather*}
1=\norm{\psi_n}_{W^{1,p}(E)}\geq n\left(\norm{\D_n\psi_n}_{L^p(E)}+\norm{\mathcal{B}\psi_n}_{W^{1-1/p,p}(E\vert_{\partial M})}+\norm{\psi_n}_{L^p(E)}\right),
\intertext{and}
\norm{\slashed\Gamma_n}_{L^{\mu}(M)}\leq1,
\end{gather*}
where $\D_n=\D_0+\slashed\Gamma_n$. Then for $\max\set{p,m}<p'<\mu$, $\slashed\Gamma_n$ is a bounded subset in $L^{p'}(E)$ and hence there exists a subsequence, denoted also by $\slashed\Gamma_n$, that converges weakly to $\slashed\Gamma\in L^{p'}(M)$ in the reflexive space $L^{p'}(E)$. We may assume that
\begin{equation*}
\psi_n\rightharpoonup\psi\quad W^{1,p}(E),\quad\psi_n\to\psi\quad L^{\tilde p}(E)
\end{equation*}
according to the Sobolev-Kondrachev embedding theorem where $1/\tilde p=1/p-1/p'>1/p-1/m$. Hence $\psi=0$ by the choice of $\psi_n$.  Moreover, if we denote $\D=\D_0+\slashed\Gamma$, then
\begin{gather*}
\norm{\D\psi_n}_{L^p(E)}\leq\norm{\D_n\psi_n}_{L^p(E)}+\norm{\left(\slashed\Gamma-\slashed\Gamma_n\right)\psi_n}_{L^p(E)}\leq\norm{\D_n\psi_n}_{L^p(E)}+\norm{\slashed\Gamma-\slashed\Gamma_n}_{L^{p'}(M)}\norm{\psi_n}_{L^{\tilde p}(E)}.
\end{gather*}
In particular,  $\D\psi_n$ converges strongly to $0$ in $L^p(E)$ and so does $\slashed\Gamma_n\psi_n$. But we know already that
\begin{equation*}
\mathcal{B}\psi_n\to 0 \quad W^{1-1/p,p}(E\vert_{\partial M}),\quad \psi_n\to 0\quad L^p(E).
\end{equation*}
The $L^p$-estimate \autoref{thm:Lp-estimate} implies that
\begin{equation*}
1=\norm{\psi_n}_{W^{1,p}(E)}\leq c(p,\slashed\Gamma)\left(\norm{\D\psi_n}_{L^p(E)}+\norm{\mathcal{B}\psi_n}_{W^{1-1/p,p}(E\vert_{\partial M})}+\norm{\psi_n}_{L^p(E)}\right)\to 0.
\end{equation*}

\

\begin{rem}\label{rem:3.4}By a similar computation,  \autoref{prop:chirality} holds for the case of a non-smooth connection $\nabla=\nabla_0+\Gamma$ with $\Gamma\in L^{\hat{m}}(M)$.
\end{rem}
\par
\subsection{$L^2$-estimate}
In this subsection, we want to prove that the solution of the Dirac equation with chiral boundary values is unique, i.e.,
\begin{equation*}
\begin{cases}
\D\psi=0,&M;\\
\mathcal{B}\psi=0,&\partial M
\end{cases}
\end{equation*}
has only the zero solution. In dimension $m=2$, we recall H\"ormander's  $L^2$-estimate method \cite{Hoermander1965l2estimate}
which was originally developed to  get the $L^2$-existence theorem for $\bar\partial$-operators on weakly pseudo-convex domains by
using Carleman-type estimates. Later, Shaw in \cite{Shaw1985l2estimate} extended this method to $\bar\partial_b$-manifolds.
Here, we use a similar idea to derive the $L^2$-estimate for the Dirac equations and use this $L^2$-estimate to get the uniqueness of solutions
of Dirac equations with chiral boundary values. In higher dimension $m>2$, we use the weak Uniqueness Continuation Property (WUCP) for
Dirac type operator to get the uniqueness. We shall then use this uniqueness to derive some useful elliptic estimates.

We will need the following Weitzenb\"ock type formula (c.f. \cite{Jost2011riemannian})
\begin{equation}\label{eq:Weitzenbock}
\D^2=-\nabla^2+\mathcal{R},
\end{equation}
where the curvature operator $\mathcal{R}$ is given by
\begin{equation*}
\mathcal{R}\coloneqq\dfrac12e_i\cdot e_j\cdot R(e_i,e_j)
\end{equation*}
and $R(X,Y)=[\nabla_X,\nabla_Y]-\nabla_{[X,Y]}$ is the curvature of the connection $\nabla=\nabla_0+\Gamma$.
\par
\begin{theorem}[Weighted Reilly formula]\label{thm:reilly-w}Let $M,E$ be as in \autoref{thm:dirac} and suppose $\Gamma\in\mathcal{D}^{p*}(E)$. Let $f$ be a smooth function, then for every $\psi\in\Gamma(E)$, we have
\begin{equation*}
\begin{split}
&\int_{\partial M}\exp(f)\left(\rin{\bar\D\psi}{\psi}+\dfrac{m-1}{2}\left(h+\vect{n}(f)\right)\norm{\psi}^2\right)+\dfrac{m-1}{m}\int_M\exp\left(f\right)\norm{\D\psi}^2\\
=&\int_{M}\exp(f)\left(\dfrac{m-1}{2}\Delta f-\dfrac{(m-1)(m-2)}{4}\norm{\nabla f}^2+\mathcal{R}_{\psi}\right)\norm{\psi}^2\\
&+\int_M\exp((1-m)f)\norm{P\left(\exp\left(\dfrac{m}{2}f\right)\psi\right)}^2.
\end{split}
\end{equation*}
Here
\begin{equation*}
\mathcal{R}_{\psi}\norm{\psi}^2=\rin{\mathcal{R}\psi}{\psi}.
\end{equation*}
\end{theorem}
\begin{rem}
\begin{enumerate}
\item If $m=2$, we have
\begin{equation*}
\begin{split}
&\int_{\partial M}\exp(f)\left(\rin{\bar\D\psi}{\psi}+\dfrac{1}{2}\left(h+\vect{n}(f)\right)\norm{\psi}^2\right)+\dfrac12\int_{M}\exp(f)\norm{\D\psi}^2\\
=&\int_{M}\exp(f)\left(\dfrac12\Delta f+\mathcal{R}_{\psi}\right)\norm{\psi}^2+\int_M\exp(-f)\norm{P\left(\exp\left(f\right)\psi\right)}^2.
\end{split}
\end{equation*}
\item If $m>2$, let $f=(1-\tau)\log u$ (with $\tau=m/(m-2)$), then
\begin{equation*}
\begin{split}
&\int_{\partial M}u^{1-\tau}\rin{\bar\D\psi}{\psi}+\dfrac{m-1}{2}\int_{\partial M}u^{-\tau}\left(h u-\dfrac{2}{m-2}\dfrac{\partial u}{\partial\vect{n}}\right)\norm{\psi}^2+\dfrac{m-1}{m}\int_{M}u^{1-\tau}\norm{\D\psi}^2\\
=&\int_{M}u^{-\tau}\left(-\dfrac{m-1}{m-2}\Delta u+\mathcal{R}_{\psi}u\right)\norm{\psi}^2+\int_Mu^{1+\tau}\norm{P\left(u^{-\tau}\psi\right)}^2.
\end{split}
\end{equation*}
\end{enumerate}
\end{rem}
\begin{proof}We only prove this theorem for smooth setting. For general case, this can be done by density theorem. Denote the twistor operator by
\begin{equation*}
P_X\psi\coloneqq\nabla_X\psi+\dfrac1mX\cdot\D\psi,\quad\forall X\in\Gamma(TM),\psi\in\Gamma(E).
\end{equation*}
This twistor operator has the following property
\begin{equation}\label{eq:twistor}
\norm{\nabla\psi}^2=\norm{P\psi}^2+\dfrac1m\norm{\D\psi}^2.
\end{equation}
In fact, since $\trace P=0$, i.e., $e_i\cdot P(e_i)=\gamma\circ P=0$, we have that
\begin{equation*}
\begin{split}
\norm{\nabla\psi}^2=&\sum_i\norm{P_{e_i}\psi-\dfrac1me_i\cdot\D\psi}^2=\norm{P\psi}^2-\dfrac2m\sum_i\rin{P_{e_i}\psi}{e_i\cdot\D\psi}+\dfrac1m\norm{\D\psi}^2\\
=&\norm{P\psi}^2+\dfrac2m\sum_i\rin{e_i\cdot P_{e_i}\psi}{\D\psi}+\dfrac1m\norm{\D\psi}^2=\norm{P\psi}^2+\dfrac1m\norm{\D\psi}^2.
\end{split}
\end{equation*}

\par
For every smooth function $f\in C^{\infty}(M)$ and every smooth spinor $\psi\in\Gamma(E)$, by using \eqref{eq:twistor} and \eqref{eq:Weitzenbock}, we have that
\begin{equation*}
\begin{split}
&\dfrac12\Delta\left(\exp(f)\norm{\psi}^2\right)\\
=&\exp(f)\left(\dfrac12\left(\Delta f+\norm{\nabla f}^2\right)\norm{\psi}^2+\dfrac12\Delta\norm{\psi}^2+2\rin{\nabla_{\nabla f}\psi}{\psi}\right)\\
=&\exp(f)\left(\dfrac12\left(\Delta f+\norm{\nabla f}^2\right)\norm{\psi}^2+2\rin{\nabla_{\nabla f}\psi}{\psi}\right)+\exp(f)\left(\norm{\nabla\psi}^2-\rin{\D^2\psi}{\psi}+\rin{\mathcal{R}\psi}{\psi}\right)\\
=&\exp(f)\left(\dfrac12\left(\Delta f+\norm{\nabla f}^2\right)\norm{\psi}^2+2\rin{P_{\nabla f}\psi}{\psi}-\dfrac2m\rin{\nabla f\cdot\D\psi}{\psi}\right)\\
&+\exp(f)\left(\norm{P\psi}^2+\dfrac1m\norm{\D\psi}^2-\rin{\D^2\psi}{\psi}+\rin{\mathcal{R}\psi}{\psi}\right)\\
=&\exp(f)\left(\dfrac12\Delta f\norm{\psi}^2+\rin{\mathcal{R}\psi}{\psi}+\dfrac{2-m}{2m}\norm{\nabla f}^2\norm{\psi}^2\right)-\rin{\D\left(\exp(f)\D\psi\right)}{\psi}\\
&+\dfrac1m\exp(f)\norm{\D\psi}^2+\dfrac{m-2}{m}\exp(f)\rin{\nabla f\cdot\D\psi}{\psi}+\exp(-f)\norm{P\left(\exp(f)\psi\right)}^2.
\end{split}
\end{equation*}
The last identity follows from the following two identities
\begin{gather*}
\begin{split}
\norm{P\left(\exp\left(f\right)\psi\right)}^2=&\exp(2f)\norm{P_{e_i}\psi+e_i(f)\psi+\dfrac1me_i\cdot\dif f\cdot\psi}^2\\
=&\exp(2f)\left(\norm{P\psi}^2+\norm{e_i(f)\psi+\dfrac1me_i\cdot\dif f\cdot\psi}^2+2\rin{P_{\nabla f}\psi}{\psi}\right)\\
=&\exp(2f)\left(\norm{P\psi}^2+\dfrac{m-1}{m}\norm{\psi}^2\norm{\nabla f}^2+2\rin{P_{\nabla f}\psi}{\psi}\right),
\end{split}
\intertext{and}
\rin{\D\left(\exp(f)\D\psi\right)}{\psi}=\exp(f)\left(\rin{\D^2\psi}{\psi}+\rin{\nabla f\cdot\D\psi}{\psi}\right).
\end{gather*}
Integrating by parts, we get that
\begin{equation*}
\begin{split}
&\int_{\partial M}\exp(f)\left(\rin{\bar\D\psi}{\psi}+\dfrac12\left((m-1)h+\vect{n}(f)\right)\norm{\psi}^2\right)\\
=&\int_{M}\exp(f)\left(\dfrac12\Delta f\norm{\psi}^2+\rin{\mathcal{R}\psi}{\psi}+\dfrac{2-m}{2m}\norm{\nabla f}^2\norm{\psi}^2\right)+\int_M\exp(-f)\norm{P\left(\exp(f)\psi\right)}^2\\
&-\dfrac{m-1}{m}\int_M\exp(f)\left(\norm{\D\psi}^2+\dfrac{m-2}{m-1}\rin{\D\psi}{\nabla f\cdot\psi}\right).
\end{split}
\end{equation*}
By using the following identity
\begin{equation*}
\begin{split}
&\norm{\D\left(\exp\left(\dfrac{m-2}{2(m-1)}f\right)\psi\right)}^2=\exp\left(\dfrac{m-2}{m-1}f\right)\norm{\D\psi+\dfrac{m-2}{2(m-1)}\nabla f\cdot\psi}^2\\
=&\exp\left(\dfrac{m-2}{m-1}f\right)\left(\norm{\D\psi}^2+\dfrac{m-2}{m-1}\rin{\D\psi}{\nabla f\cdot\psi}+\dfrac{(m-2)^2}{4(m-1)^2}\norm{\nabla f}^2\norm{\psi}^2\right)
\end{split}
\end{equation*}
we get that
\begin{equation*}
\begin{split}
&\int_{\partial M}\exp(f)\left(\rin{\bar\D\psi}{\psi}+\dfrac12\left((m-1)h+\vect{n}(f)\right)\norm{\psi}^2\right)\\
=&\int_{M}\exp(f)\left(\dfrac12\Delta f\norm{\psi}^2+\rin{\mathcal{R}\psi}{\psi}+\dfrac{2-m}{4(m-1)}\norm{\nabla f}^2\norm{\psi}^2\right)+\int_M\exp(-f)\norm{P\left(\exp(f)\psi\right)}^2\\
&-\dfrac{m-1}{m}\int_M\exp\left(\dfrac{1}{m-1}f\right)\norm{\D\left(\exp\left(\dfrac{m-2}{2(m-1)}f\right)\psi\right)}^2.
\end{split}
\end{equation*}
Set $g=f/(m-1)$ and $\sigma=\exp\left((m-2)f/(2m-2)\right)\psi$, then we have
\begin{equation*}
\begin{split}
&\int_{\partial M}\exp(g)\left(\rin{\bar\D\sigma}{\sigma}+\dfrac{m-1}{2}\left(h+\vect{n}(g)\right)\norm{\sigma}^2\right)\\
=&\int_{M}\exp(g)\left(\dfrac{m-1}{2}\Delta g-\dfrac{(m-1)(m-2)}{4}\norm{\nabla g}^2+\mathcal{R}_{\sigma}\right)\norm{\sigma}^2\\
&+\int_M\exp((1-m)g)\norm{P\left(\exp\left(\dfrac{m}{2}g\right)\sigma\right)}^2-\dfrac{m-1}{m}\int_M\exp\left(g\right)\norm{\D\sigma}^2.
\end{split}
\end{equation*}
\end{proof}
It is well known that the curvature operator $\mathcal{R}$ of $E$ can be calculated as
\begin{equation*}
\mathcal{R}=\mathcal{R}_0+\dif\Gamma+[\omega_0,\Gamma]+[\Gamma,\omega_0]+[\Gamma,\Gamma]
\end{equation*}
where $\omega_0$ is the associated connection $1$-form of $\nabla_0$ (c.f. \cite{Lawson1989spin}). In particular, $\norm{\mathcal{R}}\coloneqq\norm{\mathcal{R}}_{op}\in L^{p^*}(M)$, where the operator norm $\mathcal{R}_{op}$ at each point is defined by
\begin{equation*}
\norm{\mathcal{R}}_{op}=\sup_{\psi\neq0}\dfrac{\norm{R\psi}}{\norm{\psi}}.
\end{equation*}
 We  need the following Lemma.
\begin{lem}\label{lem:reilly-f}Suppose $M$ is a Riemann surface with boundary and $\Gamma\in\mathfrak{D}^{p^*}(E)$. There is a function $f\in W^{2,p}(M)$ for all $1<p<p^*$  satisfying
\begin{equation*}
\begin{cases}
\dfrac12\Delta f-\norm{\mathcal{R}}=0,&\text{in}\ M;\\
f=0,&\text{on}\ \partial M.
\end{cases}
\end{equation*}
\end{lem}
 Now we can state the following $L^2$-estimate in dimension $m=2$.
\begin{theorem}[$L^2$-estimate]\label{thm:L2-estimate}Let $M,E$ be as in \autoref{thm:dirac}. Suppose that $m=2$ and $\Gamma\in\mathfrak{D}^{p^*}(E)$. Then there exists a function $f\in W^{2,p}(M)$ for all $1<p<p^*$, such that for every spinor $\psi\in H^1(E)$
\begin{equation}\label{eq:l2}
\begin{split}
&\int_{\partial M}\exp(f)\left(\rin{\bar\D\psi}{\psi}+\dfrac{1}{2}\left(h+\vect{n}(f)\right)\norm{\psi}^2\right)+\dfrac{1}{2}\int_{M}\exp(f)\norm{\D\psi}^2\\
\geq&\int_{M}\exp(-f)\norm{P\left(\exp\left(f\right)\psi\right)}^2.
\end{split}
\end{equation}
\end{theorem}
\begin{proof}Choose $f$ in \autoref{lem:reilly-f}. Since $f\in W^{2,p}(M)$ for all $1<p<p^*$, by using the Sobolev embedding theorem, we know that $f\in C^{2-2/p}(M)$. In particular, $f$ is continuous.  Moreover, the trace theorem implies that
\begin{equation*}
f\vert_{\partial M}\in W^{2-1/p,p}(\partial M)\end{equation*}
and hence again the Sobolev theorem
\begin{equation*}
W^{1-1/p,p}(\partial M)\subset L^q(\partial M)
\end{equation*}
implies that $\vect{n}(f)\in L^q(\partial M)$, where we assume $1<p<\min\set{2,p^*}$ without loss of generality and
\begin{equation*}
\dfrac1q=\dfrac1p-\left(1-\dfrac1p\right)=\dfrac{2-p}{p}.
\end{equation*}
Hence
\begin{equation*}
\begin{split}
\abs{\int_{\partial M}e^f\vect{n}(f)\norm{\psi}^2}\leq C\norm{\vect{n}(f)}_{L^q(\partial M)}\norm{\psi}_{L^{q'}(\partial M)}^2\leq C\norm{\psi}_{H^{1}(M)}^2.
\end{split}
\end{equation*}
Here we have used the Sobolev embedding  and the trace theorem
\begin{equation*}
H^{1/2}(\partial M)\subset L^{q'}(\partial M),
\end{equation*}
where
\begin{equation*}
\dfrac12\left(1-\dfrac{1}{q}\right)=\dfrac{1}{q'}>\dfrac{m-2}{2(m-1)}=0
\end{equation*}
which is equivalent to
\begin{equation*}
p>1.
\end{equation*}
For the remaining terms, it is easy to get that
\begin{gather*}
\abs{\int_{\partial M}e^f\rin{\bar\D\psi}{\psi}}\leq C\norm{\psi}_{H^1(M)}^2,\\
\int_Me^f\norm{\D\psi}^2\leq C\norm{\psi}^2_{H^1(M)},
\intertext{and}
\int_Me^{-f}\norm{P\left(e^{f}\psi\right)}^2\leq C\norm{\psi}^2_{H^1(M)}.
\end{gather*}
As a consequence, by using \autoref{thm:reilly-w} and \autoref{rem:3.4} we know that \eqref{eq:l2} holds for all $\psi\in H^1(E)$ by using the density theorem.
\end{proof}
\begin{cor}[Uniqueness dimension $m=2$]\label{cor:kernel}Let $M,E$ be as in \autoref{thm:dirac}. Suppose $m=2$ and $\Gamma\in\mathfrak{D}^{p^*}(E)$, then a weak solution of
\begin{equation*}
\begin{cases}
\D\psi=0,&M;\\
\mathcal{B}\psi=0,&\partial M
\end{cases}
\end{equation*}
is a trivial spinor, i.e., $\psi=0$.
\end{cor}
\begin{proof}It follows that $\psi$ is a strong solution since $\slashed\Gamma\in L^{2p^*}(M)$, i.e., $\psi\in H^1(E)$,  according to the elliptic estimates.
By using \autoref{prop:chirality} and \autoref{rem:3.4}, we have $\psi=0$ on the boundary. \autoref{thm:L2-estimate} then implies that $P(e^{f}\psi)=0$ in $M$. That is, for every tangent vector field $X$ on $M$, we have
\begin{equation}\label{eq:twistor-w}
\nabla_X\psi+X(f)\psi+\dfrac12X\cdot\nabla f\cdot\psi=0
\end{equation}
in the weak sense. Notice that
\begin{equation*}
\rin{X\cdot\nabla f\cdot\psi}{\psi}=\rin{\psi}{\nabla f\cdot X\cdot\psi}=-\rin{\psi}{X\cdot\nabla f\cdot\psi}-2X(f)\norm{\psi}^2.
\end{equation*}
As a consequence,
\begin{equation*}
\rin{X\cdot\nabla f\cdot\psi}{\psi}=-X(f)\norm{\psi}^2.
\end{equation*}
Therefore, it follows from \eqref{eq:twistor-w} that
\begin{equation*}
\rin{\nabla_X\psi}{\psi}+\dfrac{1}{2}X(f)\norm{\psi}^2=0
\end{equation*}
which means that
\begin{equation*}
\nabla\left(e^{f}\norm{\psi}^2\right)=0
\end{equation*}
in $M$, i.e., $e^{f}\norm{\psi}^2$ is a constant in $M$. Remembering that we have proved that $\psi=0$ along the boundary, we then get that $\psi=0$ in the whole manifold $M$.
\end{proof}
\par
For higher dimension, first we have the following uniqueness theorem.
\begin{theorem}[Uniqueness for small perturbation]Let $M,E$ be as in \autoref{thm:dirac} and $m>2$.
There is a constant $\varepsilon>0$ such that if $\norm{\mathcal{R}}_{L^{m/2}}<\varepsilon$,
then there is no nontrivial solution of the following boundary value problem
\begin{equation*}
\begin{cases}
\D\psi=0,&M;\\
\mathcal{B}\psi=0,&\partial M
\end{cases}
\end{equation*}
\end{theorem}
\begin{proof}The proof is a direct consequence of the Bochner formula, the Poincar\'e-Sobolev inequality, \autoref{prop:chirality} and \autoref{rem:3.4}.
First, according to \autoref{prop:chirality}, we know that $\psi\vert_{\partial M}=0$, and then the Poincar\'e-Sobolev inequality yields
\begin{equation*}
\norm{\psi}_{L^{2m/(m-2)}(M)}\leq C_{PS}\norm{\nabla\psi}_{L^2(M)}.
\end{equation*}
Second, the classical Bochner formula (or c.f. \autoref{thm:reilly-w} with weight function $f=0$) says
\begin{align*}
0=&\int_M\rin{\mathcal{R}\psi}{\psi}+\int_M\norm{\nabla\psi}^2.
\end{align*}
Now applying the H\"older  and Poincar\'e-Sobolev inequalities, we get
\begin{align*}
0\geq&\norm{\nabla\psi}_{L^2(M)}^2-\norm{\mathcal{R}}_{L^{m/2}(M)}\norm{\psi}_{L^{2m/(m-2)}(M)}^2\\
\geq&\norm{\nabla\psi}_{L^2(M)}^2-C_{PS}\varepsilon\norm{\nabla\psi}_{L^{2}(M)}^2.
\end{align*}
Hence, if $\varepsilon<C_{PS}^{-1}$, we get $\nabla\psi=0$. Therefore, $\psi\equiv 0$ in $M$.
\end{proof}
In the general case,  we still have  uniqueness if we require more regularity on $\Gamma$. For example,  $\Gamma\in\mathfrak{D}^{(3m-2)/4}$.
To see this, we recall the weak Unique Continuation Property (WUCP) for Dirac type operators $D+V$, where $D$ is a Dirac operator with a smooth connection
 and $V$ is a potential  (see \cite{Carleman1939sur} for $V$ continuous, \cite{Booss2002weak} for $V$ bounded, and  \cite{Jerison1986carleman}
for $V \in L^{(3m-2)/2}$).  $D+V$ is said to satisfy the WUCP, if for any solution $\psi\in H^1(M)$ of $(D+V)\psi=0$ such that $\psi$ vanishes in a nonempty open
subset of $M$, then $\psi$ vanishes in the whole connected component of $M$. The proofs of the WUCP are based on certain Carleman-type estimates. For sharper results on the structure of the zero set of solutions of
generalized Dirac equations, we refer to \cite{Bar1997nodal}.

\begin{theorem}[WUCP, see \cite{Jerison1986carleman}]Let $M,E$ be as in \autoref{thm:dirac} and $m>2$.
Let $D+V$ be a Dirac type operator, where $D$ is a Dirac operator with a smooth connection and $V \in L^{(3m-2)/2}(M)$ is a potential.
Then the WUCP holds for $D+V$ .
\end{theorem}

Thanks to this WUCP, we can  apply some extension arguments  similar to the smooth case considered in \cite{Booss1993elliptic} to derive the uniqueness theorem.
\begin{theorem}[Uniqueness in dimension $m>2$]\label{thm:kernel2}Let $M,E$ be as in \autoref{thm:dirac} and $m>2$.
Suppose that $\Gamma\in\mathfrak{D}^{p^*}(E)$, $p^*\geq(3m-2)/4$.
Then there is no nontrivial solution of the following boundary value problem
\begin{equation*}
\begin{cases}
\D\psi=0,&M;\\
\mathcal{B}\psi=0,&\partial M
\end{cases}
\end{equation*}
\end{theorem}
\begin{proof}According to \autoref{prop:chirality} and \autoref{rem:3.4}, we have $\psi=0$ on the boundary.
First, there is a closed double $\tilde M$ of $M$ and a Dirac bundle $\tilde E$ on $\tilde M$ such that $\tilde E\vert_M=E$ and
\begin{equation*}
\tilde\D_0\vert_M=\D_0,
\end{equation*}
where $\tilde\D_0$ is the associated Dirac operator of $\tilde E$ (c.f. \cite{Booss1993elliptic}).
Here we write $\D=\D_0+\slashed\Gamma$ and $\D_0$ is smooth.
Extend $\slashed\Gamma$  trivially to some $\tilde{\slashed\Gamma}$ on $\tilde M$, i.e.,
\begin{equation*}
\tilde\Gamma=\begin{cases}
\Gamma,&M;\\
0,&\tilde M\setminus M.
\end{cases}
\end{equation*}
Then the trivial extension $\tilde\psi$ of $\psi$, i.e.
\begin{equation*}
\tilde\psi=\begin{cases}
\psi,&M;\\
0,&\tilde M\setminus M,
\end{cases}
\end{equation*}
is a $H^1(\tilde M)$-solution of
\begin{equation*}
\tilde\D_0\tilde\psi+\tilde{\slashed\Gamma}\tilde\psi=0.
\end{equation*}
We need only to check that $\tilde\psi$ is a weak solution. For every smooth spinor $\varphi$ on $\tilde M$, we have
\begin{align*}
\int_{\tilde M}\hin{\tilde\psi}{\tilde\D^*_0\varphi+\tilde{\slashed\Gamma}^*\varphi}=&\int_{M}\hin{\psi}{\tilde\D^*_0\varphi+\slashed\Gamma^*\varphi}\\
=&\int_M\hin{\D_0\psi+\slashed\Gamma\psi}{\varphi}-\int_{\partial M}\hin{\sigma_{\vect{n}}(\tilde\D_0)\psi}{\varphi}\\
=&0.
\end{align*}
Now we can apply the weak UCP of $\tilde\D_0+\tilde{\slashed\Gamma}$ to show that $\tilde\psi=0$ in the whole manifold $\tilde M$. Therefore, $\psi\equiv 0$ in  $M$.
\end{proof}
Now we can state the main elliptic $L^p$-estimates.
\begin{theorem}[Main $L^p$-estimate]\label{thm:main-estimate}Let $M,E$ be as in \autoref{thm:dirac} and $m\geq2$.
Suppose that $\Gamma\in\mathfrak{D}^{p^*}(E)$. Then for $1<p<p^*$, there exists a constant $c=c(p,\Gamma)>0$ such that for any $\psi \in W^{1,p}(E)$
\begin{equation*}
\norm{\psi}_{W^{1,p}(E)}\leq c\left(\norm{\D\psi}_{L^p(E)}+\norm{\mathcal{B}\psi}_{W^{1-1/p,p}(E\vert_{\partial M})}\right).
\end{equation*}
\end{theorem}
\begin{proof}Consider the operator
\begin{equation*}
(\D,\mathcal{B}):W^{1,p}(E)\To L^p(E)\times W^{1-1/p,p}(E\vert_{\partial M}).
\end{equation*}
Since $\Gamma\in L^{2p^*}(E)$, this  is well defined. Moreover, by using \autoref{thm:Lp-estimate}, we have the following $L^p$-estimates
\begin{equation*}
\norm{\psi}_{W^{1,p}(E)}\leq c\left(\norm{\D\psi}_{L^p(E)}+\norm{\mathcal{B}\psi}_{W^{1-1/p,p}(E\vert_{\partial M})}+\norm{\psi}_{L^p(E)}\right),
\end{equation*}
where $c=c(p,\Gamma)>0$.
\par
Now one can show that the range of $(\D,\mathcal{B})$ is closed and the kernel is trivial. In fact, the kernel is obviously trivial by using \autoref{cor:kernel} (for $m=2$) and \autoref{thm:kernel2} (for $m>2$). Now we prove that the image is closed. For $\psi_n\in W^{1,p}(E)$ with
\begin{equation*}
\D\psi_n\to\varphi\in L^p(E),\quad \mathcal{B}\psi_n\to\psi_0\in W^{1-1/p,p}(E\vert_{\partial M}).
\end{equation*}
It is clear that $\mathcal{B}\psi_0=\psi_0$.
\par
 First, we assume that $\norm{\psi_n}_{L^p(E)}\leq 1$. Then the $L^p$-estimate \autoref{thm:Lp-estimate} implies that $\psi_n$ is bounded in $W^{1,p}(E)$. Hence, there exists a subsequence of $\psi_n$, denoted also by $\psi_n$, such that $\psi_n$ converges weakly to $\psi$ in $W^{1,p}(E)$ and strongly in $L^p(E)$, i.e.,
 \begin{equation*}
 \psi_n\rightharpoonup\psi\in W^{1,p}(E),\quad \psi_n\to\psi\in L^p(E).
 \end{equation*}
 Using \autoref{thm:Lp-estimate} again, $\psi_n$ is a Cauchy sequence in $W^{1,p}(E)$. As a consequence, there exists a limit of $\psi_n$ in $W^{1,p}(E)$ and this limit must be $\psi$. Hence $\varphi=\D\psi$ and $\psi_0=\mathcal{B}\psi$.
 \par
 Second, if $\psi_n$ is not bounded in $L^p(E)$, setting
\begin{equation*}
\tilde\psi_n=\dfrac{\psi_n}{\norm{\psi_n}_{L^p(E)}}\in W^{1,p}(E).
\end{equation*}
Then
\begin{equation*}
\D\tilde\psi_n\to0\in L^{p}(E),\quad\mathcal{B}\tilde\psi_n\to 0\in W^{1-1/p,p}(E\vert_{\partial M}).
\end{equation*}
 By the same arguments as above, $\tilde\psi_n$ has a limit $\tilde\psi$ in $W^{1,p}(E)$ such that $\norm{\tilde\psi}_{L^p(E)}=1,\D\tilde\psi=0$ and $\mathcal{B}\tilde\psi=0$. This is impossible since \autoref{cor:kernel} (for $m=2$) and \autoref{thm:kernel2} (for $m>2$) implies that $\tilde\psi=0$.
\par
Hence, the closed graph theorem implies that $(\D,\mathcal{B})$ is an isometry between $H^1(E)$ and the range of $(\D,\mathcal{B})$. As a consequence, we have the following estimate
\begin{equation*}
\norm{\psi}_{W^{1,p}(E)}\leq c\left(\norm{\D\psi}_{L^p(E)}+\norm{\mathcal{B}\psi}_{W^{1-1/p,p}(E\vert_{\partial M})}\right),
\end{equation*}
where $c=c(p,\Gamma)>0$.
\end{proof}
\begin{rem}If $p=2$, we can prove this theorem directly by using  \autoref{thm:Lp-estimate} and \autoref{thm:L2-estimate}. First, according to \autoref{thm:Lp-estimate}, there exists a constant $c=c(\Gamma)>0$ such that
\begin{equation*}
\norm{\psi}_{H^1(E)}\leq c\left(\norm{\D\psi}_{L^2(E)}+\norm{\mathcal{B}\psi}_{H^{1/2}(E\vert_{\partial M})}+\norm{\psi}_{L^2(E)}\right)
\end{equation*}
holds for all $\psi\in H^1(E)$. Second, \autoref{thm:L2-estimate} implies that
\begin{equation*}
\norm{\psi}_{L^2(E)}\leq c(\Gamma)\left(\norm{\D\psi}_{L^2(E)}+\norm{\mathcal{B}\psi}_{H^{1/2}(E\vert_{\partial M})}\right).
\end{equation*}
Combining these two estimates, we complete the proof.
\end{rem}
\begin{rem}\label{rem: norm-gamma}We can choose  $c=c\left(p,\norm{\slashed\Gamma}_{p^*}\right)>0$. See Remark \ref{rem:mu>2}.
\end{rem}
\par
\subsection{Existence and uniqueness for solutions of Dirac equations}
In this subsection, we shall consider the existence and uniqueness of solutions of the Dirac equation with chiral boundary conditions, to find a solution $\psi\in W^{1,p}(E)$ of the following,
\begin{equation}\label{eq:Dirac}
\begin{cases}
\D\psi=\varphi,&M;\\
\mathcal{B}\psi=\mathcal{B}\psi_0,&\partial M,
\end{cases}
\end{equation}
where $\varphi\in L^p(E),\mathcal{B}\psi_0\in W^{1-1/p,p}(E\vert_{\partial M})$.
\par
 Several  general existence theorems for this system in $H^1(E)$ have been derived  under some integral conditions, for example, see \cite{Bartnik2005boundary}, p.53, Theorem 7.3. which asserts that \eqref{eq:Dirac} is solvable in $H^1(E)$ if and only if the following integral condition holds,
\begin{equation*}
\int_{M}\hin{\varphi}{\eta}=0,\quad\forall\eta\in\ker(\D,\mathcal{B}^*).
\end{equation*}
Moreover, this solution satisfies the following $L^2$-estimate
\begin{equation*}
\norm{\psi}_{H^1(E)}\leq c\left(\norm{\varphi}_{L^2(E)}+\norm{\mathcal{B}\psi_0}_{H^{1/2}(E\vert_{\partial M})}+\norm{\psi}_{L^2(E)}\right).
\end{equation*}
But the   uniqueness may  not be true for a general first order elliptic partial differential equation with an elliptic boundary condition.
\par
 Notice that in our setting, this integral condition is always satisfied for each $\varphi\in L^2(E)$ since the kernel of $(\D,\mathcal{B}^*)$ is zero according to \autoref{cor:kernel} (for $m=2$) and \autoref{thm:kernel2} (for $m>2$). In fact, in our setting, we can state the  existence and uniqueness \autoref{thm:dirac} with the help of the main $L^p$-estimate \autoref{thm:main-estimate}.

\begin{proof}[\textbf{Proof of \autoref{thm:dirac}}]  We only need to show the existence. We can use a  method that is similar to that for  deducing the analogous theorem for second order elliptic partial differential equations with Dirichlet boundary values, see \cite{Gilbarg2001elliptic}, p.241, Theorem 9.15, for example.  For  convenience, we will give a detailed proof here.
\par
 First, we consider the case $p^*>2$ and $p=2$. The following argument is typical, see \cite{Bartnik2005boundary,Gilbarg2001elliptic} for example. Let us consider the following closed subspace of $H^1(E)$,
\begin{equation*}
H^1_{\mathcal{B}}(E)=\set{\psi\in H^1(E):\mathcal{B}\psi=0}.
\end{equation*}
\autoref{thm:main-estimate} gives the a-priori estimate
\begin{equation*}
\norm{\psi}_{H^1(E)}^2=\int_M\norm{\nabla_0\psi}^2+\norm{\psi}^2\leq C\int_{M}\norm{\D\psi}^2,\quad\forall\psi\in H^1_{\mathcal{B}}(E).
\end{equation*}
 In particular, $\int_M\norm{\D\psi}^2$ is strictly coercive on $H^1_{\mathcal{B}}(E)$, so the Lax-Milgram theorem gives $\psi\in H^1_{\mathcal{B}}(E)$ satisfying
 \begin{equation*}
 \int_{M}\hin{\varphi}{\D\eta}=\int_{M}\hin{\D\psi}{\D\eta},\quad\eta\in H^1_{\mathcal{B}}(E).
 \end{equation*}
 Denote $\Phi=\D\psi-\varphi\in L^2(E)$, then
 \begin{equation*}
 \int_M\hin{\Phi}{\D\eta}=0,\quad\forall\eta\in H^1_{\mathcal{B}}(E).
 \end{equation*}
 Therefore $\Phi$ is a weak solution of
 \begin{equation*}
 \D\Phi=0,\quad\mathcal{B}^*\Phi=0.
 \end{equation*}
 Since $\mathcal{B}^*$ is elliptic, so all the elliptic estimates stated for $\mathcal{B}$ can be stated in a similar way (see \autoref{thm:regularity_weak}). In particular, $\Phi$ is a strong solution, i.e., $\Phi\in H^1(E)$. By using the $L^2$-estimate of \autoref{thm:L2-estimate} (for $\mathcal{B}^*$), we know that $\Phi=0$. Hence $\D\psi=\varphi$ and $\mathcal{B}\psi=0$.
 \par
 In the general case, we extend $\mathcal{B}\psi_0$ to a spinor $\tilde\psi\in H^1(E)$ such that $\tilde\psi\vert_{\partial M}=\mathcal{B}\psi_0$. Setting $\hat\psi=\psi-\tilde\psi$, we then have
 \begin{equation*}
 \begin{cases}
 \D\hat\psi=\varphi-\D\tilde\psi,&M;\\
 \mathcal{B}\hat\psi=0,&\partial M.
 \end{cases}
 \end{equation*}
 The previous case shows that there is a solution $\hat\psi\in H^1(E)$. Then $\psi=\hat\psi+\tilde\psi$ is the desired solution of \eqref{eq:Dirac}.
 \par
 Second, we consider the case $1<p<p^*$. Let $\varphi_{\varepsilon}\in\Gamma(E)$ such that $\varphi_{\varepsilon}$ converges strongly to $\varphi$ in $L^p(E)$ as $\varepsilon\to0$. For each $\varepsilon>0$, let $\psi_{\varepsilon}$ be the unique solution of
 \begin{equation*}
 \D\psi_{\varepsilon}=\varphi_{\varepsilon},\quad\mathcal{B}\psi_{\varepsilon}=0.
 \end{equation*}
 The a-priori estimate of  \autoref{thm:main-estimate} says that $\psi_{\varepsilon}\in W^{1,p}(E)$ and is a Cauchy sequence in $W^{1,p}(E)$ since $\varphi_{\varepsilon}$ converges strongly to $\varphi$ in $L^p(E)$. Say $\psi_{\varepsilon}$ converges strongly to $\psi$ in $W^{1,p}(E)$, then $\D\psi=\varphi$ and $\mathcal{B}\psi=0$. In the non-homogeneous boundary case, we set $\psi=\hat\psi+\tilde\psi$ where $\tilde\psi\in W^{1,p}(E)$ is an extension of $\mathcal{B}\psi_0$ such that $\tilde\psi\vert_{\partial M}=\mathcal{B}\psi_0$ by the extension theorem. This is possible since $1<p<\infty$. Then we choose a solution $\hat\psi\in W^{1,p}(E)$ such that
 \begin{equation*}
 \begin{cases}
 \D\hat\psi=\varphi-\D\tilde\psi,&M;\\
 \mathcal{B}\hat\psi=0,&\partial M.
 \end{cases}
 \end{equation*}
 Now we get a solution of \eqref{eq:Dirac}.
\end{proof}
\par

\vskip0.2cm

\section{Dirac equations along a map}\label{sec:dirac-map}
We first consider the following system which is slightly more general than  \eqref{eq:Dirac1}:
\begin{equation}\label{eq:system}
\begin{cases}
\D\psi^A+\Omega^A_B\cdot\psi^B=\eta^A,&M;\\
\mathcal{B}\psi^A=\mathcal{B}\psi_0^A,&\partial M,
\end{cases}
\end{equation} $A=1,\cdots, q,$
here $\eta^A\in L^p(E),\mathcal{B}\psi^A_0\in W^{1-1/p,p}(E\vert_{\partial M})$ and $\Omega\in\Omega^1\left(\La{so}_n\right)$, i.e., $\Omega^A_B=-\Omega^B_A$. Under suitable conditions, we can solve this system.
\begin{theorem}\label{thm:dirac-omega}Let $M,E$ be as in \autoref{thm:dirac}. Suppose that $\Gamma\in\mathfrak{D}^{p^*}(E),\dif\Omega\in L^{p^*}(M),\Omega\in L^{2p^*}(M)$. Then for $1<p<p^*$, \eqref{eq:system} admits a unique solution. Moreover, we have the following elliptic estimate
\begin{equation*}
\norm{\psi}_{W^{1,p}(E)}\leq c\left(\norm{\eta}_{L^p(E)}+\norm{\mathcal{B}\psi_0}_{W^{1-1/p,p}(E\vert_{\partial M})}\right),
\end{equation*}
where $c=c\left(p,\norm{\Gamma}_{p^*},\norm{\Omega}_{L^{2p^*}(M)}+\norm{\dif\Omega}_{L^{p^*}(M)}\right)>0$.
\end{theorem}
\begin{proof}We construct a new Dirac bundle and a new chirality operator, and then we can apply the existence and uniqueness for the usual Dirac equation to prove this theorem.
\par
Let $\tilde E=\oplus^nE=\underbrace{E\oplus\dotsc\oplus E}_{\text{n times}}$. Then $\tilde E$ becomes a Dirac bundle  as a Whitney sum bundle. The Clifford multiplication is defined as
\begin{equation*}
\tilde\gamma(X)(\psi^A)\coloneqq(X\cdot\psi^A),\forall X\in TM,
\end{equation*}
i.e., $\tilde\gamma=\gamma^E\Id$. Here $\gamma^E(X)\psi^A\coloneqq X\cdot\psi^A$ stands for the Clifford multiplication on $E$.
Then the associated $\tilde\Gamma=\Gamma\Id$, i.e.,
\begin{equation*}
\tilde\Gamma(X)(\psi^A)\coloneqq(\Gamma(X)\psi^A).
\end{equation*}
\par
Define $\Gamma'$ by
\begin{equation*}
\Gamma'(X)(\psi^A)\coloneqq(\Omega^A_B(X)\psi^B).
\end{equation*}
It is clear that $\Gamma'\in\Omega^1(\Lg{Ad}(\tilde E))$. We need only to check that $[\Gamma',\tilde\gamma]=0$ in order to prove that $\Gamma'\in\mathfrak{D}^{p^*}(\tilde E)$. In fact,
\begin{equation*}
[\Gamma',\tilde\gamma](X,Y)(\psi^A)=\left(\Omega^A_B(X)\left(Y\cdot\psi^B\right)-Y\cdot\Omega^A_B(X)\psi^B\right)=0.
\end{equation*}
Therefore, we have constructed a new Dirac bundle $\tilde E$ with the Dirac operator $\tilde D$ defined by
\begin{equation*}
\tilde\D(\psi^A)=(\D\psi^A+\Omega^A_B\cdot\psi^B).
\end{equation*}
It is obvious that $\dif(\tilde\Gamma+\Gamma')\in L^{p^*}(M),\tilde\Gamma+\Gamma'\in L^{2p^*}(M)$.
\par
Introduce an operator $\tilde G\in\hom(\tilde E)$,
\begin{equation*}
\tilde G(\psi^A)\coloneqq(G\psi^A).
\end{equation*}
It is clear that
\begin{equation*}
\tilde G^2=\Id,\quad \tilde G^*=\tilde G,\quad \tilde G\tilde\gamma(X)=-\tilde\gamma(X)\tilde G,\quad\forall X\in TM.
\end{equation*}
Moreover, $\tilde\nabla\tilde G=0$. In fact,
\begin{equation*}
\tilde\nabla_X\tilde G(\psi^A)=\left(\nabla_XG\psi^A+\Omega^A_B(X)G\psi^B\right)=\left(G\nabla_X\psi^A+G\Omega^A_B(X)\psi^B\right)=\tilde G\tilde\nabla_X(\psi^A).
\end{equation*}
In particular, $\tilde G$ is a chirality operator on $\tilde E$. Therefore, the associated chirality boundary operator
\begin{equation*}
\mathcal{\tilde B}(\psi^A)\coloneqq\left(\mathcal{B}\psi^A\right).
\end{equation*}
 Then we can use the theory of the Dirac equation \autoref{thm:dirac} to finish the proof of this theorem.
\end{proof}
\par
Now we suppose that $M$ is a Riemannian spin manifold with boundary $\partial M$ and give the
\begin{proof}[\textbf{Proof of  \autoref{thm:main-dirac}}]Embedding $N$ into some Euclidian space, then as shown in \autoref{sec:dh}, we can rewrite this boundary value problem for the Dirac equation as
\begin{equation*}
\begin{cases}
\dirac\Psi^A+\Omega^A_B\cdot\Psi^B=\eta^A,&M;\\
\mathcal{B}\Psi^A=\mathcal{B}\psi^A,&\partial M,
\end{cases}
\end{equation*}
where
\begin{equation*}
\Omega^A_B=[\nu(\Phi),\dif\nu(\Phi)]^A_B.
\end{equation*}
In particular, $\dif\Omega=[\dif\nu(\Phi),\dif\nu(\Phi)]$. Therefore, if $\Phi\in W^{1,2p^*}(M,N)$, then $\Phi\in C^0(\bar M,N)$ by using the Sobolev embedding theorem. As a consequence,
\begin{equation*}
\Omega\in L^{2p^*}(M),\quad \dif\Omega\in L^{p^*}(M).
\end{equation*}
Hence, by using the \autoref{thm:dirac-omega}, we get a unique solution of $\Psi\in W^{1,p}(\Sigma M\otimes\Phi^{-1}T\R^q)$ for some larger $q$. Moreover, there exists a constant $c=c\left(p,\norm{\Phi}_{W^{1,2p^*}(M)}\right)>0$ such that
\begin{equation*}
\norm{\Psi}_{W^{1,p}(M)}\leq c\left(\norm{\eta}_{L^p(M)}+\norm{\mathcal{B}\psi}_{W^{1-1/p,p}(\partial M)}\right).
\end{equation*}
\par
Now we want to prove that $\Psi$ is a spinor along the map $\Phi$. Introduce $\tilde\Psi^A=\nu^A_B\Psi^B$, then we need only to prove that $\tilde\Psi=0$.
\begin{claim}
\begin{equation*}
\begin{cases}
\dirac\tilde\Psi^A+\Omega^A_B\cdot\tilde\Psi^B=0,& M;\\
\mathcal{B}\tilde\Psi^A=0,&\partial M.
\end{cases}
\end{equation*}
\end{claim}
If this claim is true, then using the \autoref{thm:dirac-omega} again, we get that $\tilde\Psi=0$. Hence, we complete the proof of this theorem.
\par
Now, we confirm the claim. In fact, noticing that $\nu^A_B\eta^B=0,\nu^A_B\mathcal{B}\psi^B=0$ since $\eta$ is a spinor along the map $\Phi$ and $\mathcal{B}\psi$ is the restriction of a spinor along the map $\Phi$ to the boundary $\partial M$, we have that
\begin{gather*}
\begin{split}
\dirac\tilde\Psi^A=&\nu^A_B\dirac\psi^B+\dif\nu^A_B\cdot\Psi^B=-\nu^A_B\dif\nu^B_C\cdot\Psi^C+\nu^A_B\dif\nu^B_C\nu^C_D\cdot\Psi^D+\nu^A_B\eta^B+\dif\nu^A_B\cdot\Psi^B\\
=&\dif\nu^A_B\cdot\tilde\Psi^B=-\dif\pi^A_B\cdot\tilde\Psi^B=\left(\dif\pi^A_C\pi^C_B-\pi^A_C\dif\pi^C_B\right)\cdot\tilde\Psi^B=\left(\dif\nu^A_C\nu^C_B-\nu^A_C\dif\nu^C_B\right)\cdot\tilde\Psi^B\\
=&-\Omega^A_B\cdot\tilde\Psi^B,
\end{split}
\intertext{and}
\mathcal{B}\tilde\Psi^A=\nu^A_B\mathcal{B}\Psi^B=\nu^A_B\mathcal{B}\psi^B=0.
\end{gather*}
\end{proof}
\begin{rem}Using the same method, one can prove that
\begin{equation*}
\begin{cases}
\D\psi=\eta\in L^p\left(E\otimes\Phi^{-1}V\right),&M;\\
\mathcal{B}\Psi=\mathcal{B}\psi\in W^{1-1/p,p}\left(\left(E\otimes\Phi^{-1}V\right)\vert_{\partial M}\right),&\partial M
\end{cases}
\end{equation*}
admits a unique solution $\Psi\in W^{1,p}(E\otimes\Phi^{-1}V)$, where $E$ is a Dirac bundle on $M$, $V$ is a Hermitian metric vector bundle on $N$, $\D$ is the associated Dirac operator of $E\otimes\Phi^{-1}V$, the Dirac connection $\nabla=\nabla_0+\Gamma$ satisfies the condition $\Gamma\in\mathfrak{D}^{p^*}E$ and $\Phi\in W^{1,2p^*}(M;N)$.
\end{rem}
\begin{rem}If $\Phi$ is smooth, then $\Sigma M\otimes\Phi^{-1}TN$ is a smooth Dirac bundle, in this case,  \autoref{thm:main-dirac} is just a direct corollary of \autoref{thm:dirac}.
\end{rem}
\vspace{0.2cm}
\par
Now let us give some further remarks on {\it the Schauder theory of Dirac equations.}
The interior Schauder estimate  for the Dirac equation is
\begin{theorem}[See \cite{Ammann2003variational}]Let $M,E$ be as in \autoref{thm:dirac}. Suppose that $\slashed\Gamma\in C^{\alpha}(M)$ for some $0<\alpha<1$, then for all $M'\Subset M''\Subset M$,
\begin{equation*}
\norm{\psi}_{1+\alpha;M'}\leq c\left(\alpha,\dist(M',\partial M''),\norm{\slashed\Gamma}_{\alpha;M''}\right)\left(\norm{\D\psi}_{\alpha;M''}+\norm{\psi}_{0;M''}\right).
\end{equation*}
\end{theorem}
Due to \autoref{prop:chirality}, one can also state a  boundary Schauder estimate for the Dirac equation.
\begin{theorem}\label{thm:schauder}Let $M,E$ be as in \autoref{thm:dirac}. Suppose that $\slashed\Gamma\in C^{\alpha}(\bar M)$ for some $0<\alpha<1$, then
\begin{equation*}
\norm{\psi}_{1+\alpha;M}\leq c\left(\alpha,\norm{\slashed\Gamma}_{\alpha:M}\right)\left(\norm{\D\psi}_{\alpha;M}+\norm{\mathcal{B}\psi}_{1+\alpha;\partial M}+\norm{\psi}_{0;M}\right).
\end{equation*}
\end{theorem}
\begin{proof}The classical argument for Schauder estimates (see \cite{Gilbarg2001elliptic}) can be combined with \autoref{prop:chirality}.
\end{proof}
By using the main $L^p$-estimates of \autoref{thm:dirac}, we can get
\begin{theorem}Let $M,E$ be as in \autoref{thm:dirac}. Suppose that $\Gamma\in\mathfrak{D}^1(E),\Gamma\in C^{\alpha}(\bar M)$ and $\dif\Gamma\in C^{\alpha}(\bar M)$ for some $0<\alpha<1$, then
\begin{equation*}
\norm{\psi}_{1+\alpha;M}\leq c\left(\alpha,\norm{\Gamma}_{\alpha:M}+\norm{\dif\Gamma}_{\alpha;M}\right)\left(\norm{\D\psi}_{\alpha;M}+\norm{\mathcal{B}\psi}_{1+\alpha;\partial M}\right).
\end{equation*}
\end{theorem}
\begin{proof}By using the main $L^p$-estimates of \autoref{thm:dirac}, we know that for large $p$
\begin{equation*}
\begin{split}
\norm{\psi}_{0;M}\leq &c(p,\norm{\Gamma}_{p^*})\left(\norm{\D\psi}_{L^p(E)}+\norm{\mathcal{B}\psi}_{W^{1-1/p,p}(E\vert_{\partial M})}\right)\\
\leq& c\left(\alpha,\norm{\Gamma}_{\alpha:M}+\norm{\dif\Gamma}_{\alpha;M}\right)\left(\norm{\D\psi}_{\alpha;M}+\norm{\mathcal{B}\psi}_{1+\alpha;\partial M}\right).
\end{split}
\end{equation*}
Then applying \autoref{thm:schauder}, we prove the desired result.
\end{proof}
Similarly to the case of the $L^p$-estimate, we can prove the following two theorems:
\begin{theorem}Let $M,E$ be as in \autoref{thm:dirac}. Suppose that $\Gamma\in\mathfrak{D}^1(E),\Gamma\in C^{\alpha}(\bar M),\dif\Gamma\in C^{\alpha}(\bar M),\Omega\in C^{\alpha}(\bar M),\dif\Omega\in C^{\alpha}(\bar M)$ for some $0<\alpha<1$.
Let $\eta\in C^{\alpha}(\bar M),\mathcal{B}\psi_0\in C^{1,\alpha}(\partial M)$, then \eqref{eq:system} admits a unique solution $\psi\in C^{1,\alpha}(\bar M)$. Moreover, the following estimate holds
 \begin{equation*}
\norm{\psi}_{1+\alpha;M}\leq c\left(\alpha,\norm{\Gamma}_{\alpha;M}+\norm{\dif\Gamma}_{\alpha;M},\norm{\Omega}_{\alpha;M}+\norm{\dif\Omega}_{\alpha;M}\right)\left(\norm{\D\psi}_{\alpha;M}+\norm{\mathcal{B}\psi}_{1+\alpha;\partial M}\right).
\end{equation*}
 \end{theorem}

 \begin{theorem}Let $M,N$ be as in \autoref{thm:main-dirac}. Let $\Phi\in C^{1,\alpha}(\bar M,N)$ for some $0<\alpha<1$, then the following Dirac equation
\begin{equation*}
\begin{cases}
\D\Psi=\eta\in C^{\alpha}(\bar M;\Sigma M\otimes\Phi^{-1}TN),&M;\\
\mathcal{B}\Psi=\mathcal{B}\psi\in C^{1,\alpha}\left(\partial M;\Sigma M\otimes\Phi^{-1}TN\right),&\partial M
\end{cases}
\end{equation*}
admits a unique solution $\Psi\in C^{1,\alpha}(\bar M:\Sigma M\otimes\Phi^{-1}TN)$, where $\D$ is the Dirac operator along the map $\Phi$. Moreover,
\begin{equation*}
\norm{\Psi}_{\alpha;M}\leq c(\alpha,\norm{\Phi}_{1+\alpha;M})\left(\norm{\eta}_{\alpha;M}+\norm{\mathcal{B}\psi}_{1+\alpha;\partial M}\right).
\end{equation*}
\end{theorem}

\par

\vskip0.2cm

\section{Short time existence of first order Dirac-harmonic map flows}\label{sec:dhf}
In this section, we assume that $M^m$ ($m\geq2$) is a compact Riemannian spin manifold with boundary $\partial M$ and choose  a fixed spin structure on $M$.
\par
 Let us consider the family of coupled system of differential equations for a map $\Phi:M\times[0,T]\To\R^q$ with $\Phi=(\Phi^A)$ and for a spinor field $\Psi:M\times[0,T]\To\Sigma M\otimes\Phi^{-1}T\R^q$  with $\Psi=(\Psi^A)$ along the map $\Phi$
\begin{gather}
\begin{cases}
\left(\dfrac{\partial}{\partial t}-\Delta\right)\Phi^A+\Omega^A_B\cdot\dif\Phi^B+\hin{\tilde\Omega^A_B}{\dif\Phi^B}=0,&M\times(0,T];\\
\dirac\Psi^A+\Omega^A_B\cdot\Psi^B=0,&\partial M\times[0,T]
\end{cases}\label{eq:DHF}
\end{gather}
with the initial and boundary conditions
\begin{gather}\label{eq:BDHF}
\begin{cases}
\Phi(x,t)=\phi(x,t),&(x,t)\in\partial M\times[0,T]\cup M\times\set{0};\\
\mathcal{B}\Psi=\mathcal{B}\psi,&\partial M\times[0,T],
\end{cases}
\end{gather}
where $\mathcal{B}$ is a chirality boundary operator.

The following two Lemmas are similar to the harmonic map heat flow (c.f. \cite{Eells1964harmonic,Hamilton1975harmonic,Li1991heat}).
\begin{lem}\label{lem:equivalent}
Suppose the image of $\Phi$ lies in $N$ and $\Psi$ is a spinor along the map $\Phi$, then $(\Phi,\Psi)$ satisfies the Dirac-harmonic map flow \eqref{eq:dhf}, i.e.,
\begin{equation*}
\begin{cases}
\partial_t\Phi=\tau(\Phi)-\mathcal{R}(\Phi,\Psi),\\
\D\Psi=0,
\end{cases}
\end{equation*}
if and only if $(\Phi,\Psi)$ satisfies \eqref{eq:DHF}.
\end{lem}
\begin{proof}A well-known computation.
\end{proof}
\begin{lem}\label{lem:stational}Suppose that $(\Phi,\Psi)$ is a solution of \eqref{eq:DHF} which is continuous on $M\times[0,T]$ with $\phi(x,t)\in N$ for all $(x,t)\in\partial M\times[0,T]\cup M\times\set{0}$ and $\psi$ is a spinor along the map $\phi\vert_{\partial M}$ for all time $[0,T]$. Suppose $\Phi(x,t)\in\tilde N$ on $M\times(0,T]$, then $\Phi(x,t)\in N$ for all $(x,t)\in M\times[0,T]$ and $\Psi(\cdot,t)$ is a spinor along the map $\Phi(\cdot,t)$ for all time $t\in[0,T]$. In fact, $\tilde\Psi^A=\nu^A_B\Psi^B$ satisfies the following Dirac-type equation
\begin{equation*}
\begin{cases}
\dirac\tilde\Psi^A+\Omega^A_B\cdot\tilde\Psi^B=0,& M;\\
\mathcal{B}\tilde\Psi=0,&\partial M.
\end{cases}
\end{equation*}
\end{lem}
\begin{proof}
For $z\in\R^q$, let us define $\rho:\R^q\To\R^q$ by $\rho(z)=z-\pi(z)$. Consider
\begin{equation*}
\varphi(x,t)=\norm{\rho(\Phi(x,t))}^2=\sum_{A=1}^q\norm{\rho^A(\Phi(x,t))}^2.
\end{equation*}
We can get that
\begin{equation*}
\begin{split}
&\left(\dfrac{\partial}{\partial t}-\Delta\right)\varphi(x,t)=-2\norm{\nabla\rho(\Phi(x,t))}^2+2\hin{\partial_t\rho-\Delta\rho}{\rho}\\
=&-2\norm{\nabla\rho(\Phi(x,t))}^2+2\hin{\nu^A_B\left(\partial_t\Phi^B-\Delta\Phi^B\right)+\pi^A_{BC}\hin{\nabla\Phi^B}{\nabla\Phi^C}}{\rho^A}\\
=&-2\norm{\nabla\rho(\Phi(x,t))}^2-2\hin{\nu^A_B(\Phi)\left(\Omega(\Phi)^B_C\cdot\dif\Phi^C+\hin{\tilde\Omega(\Phi)^B_C}{\dif\Phi^C}\right)-\pi^A_{BC}(\Phi)\hin{\nabla\Phi^B}{\nabla\Phi^C}}{\rho(\Phi)^A}.
\end{split}
\end{equation*}
Notice that $\pi^A_B$ is a projection when restricted on $N$, then we obtain that restricted on $N$,
\begin{align*}
\nu^A_B(\Phi)\left(\Omega(\Phi)^B_C\cdot\dif\Phi^C+\hin{\tilde\Omega(\Phi)^B_C}{\dif\Phi^C}\right)-\pi^A_{BC}(\Phi)\hin{\nabla\Phi^B}{\nabla\Phi^C}=0.
\end{align*}
By mean value theorem, it follows that
\begin{align*}
-2\hin{\nu^A_B(\Phi)\left(\Omega(\Phi)^B_C\cdot\dif\Phi^C+\hin{\tilde\Omega(\Phi)^B_C}{\dif\Phi^C}\right)-\pi^A_{BC}(\Phi)\hin{\nabla\Phi^B}{\nabla\Phi^C}}{\rho(\Phi)^A}\leq c\varphi.
\end{align*}
Therefore, we have
\begin{equation*}
\left(\dfrac{\partial}{\partial t}-\Delta\right)\varphi\leq c\varphi.
\end{equation*}
Since $\varphi\geq0$ and $\varphi=0$ on $\partial M\times[0,T]\cup M\times\set{0}$, we have that $\varphi=0$ on $M\times[0,T]$. Hence $\Phi(x,t)\in N$ for all $(x,t)\in M\times[0,T]$ according to the maximum principle.
\par
Next we show that $\Psi$ is a spinor along the map $\Phi$. In order to do this, we consider $\tilde\Psi^A=\nu^A_B\Psi^B$, then
\begin{equation*}
\begin{split}
\dirac\tilde\Psi^A=&\nu^A_B\dirac\Psi^B+\nabla\nu^A_B\cdot\Psi^B=-\nu^A_B\nabla\nu^B_C\cdot\Psi^C+\nabla\nu^A_B\cdot\Psi^B\\
=&\nabla\nu^A_B\cdot\tilde\Psi^B=-\nabla\pi^A_B\cdot\tilde\Psi^B\\
=&\left(\nabla\pi^A_C\pi^C_B-\pi^A_C\nabla\pi^C_B\right)\cdot\tilde\Psi^B=-\Omega^A_B\cdot\tilde\Psi^B.
\end{split}
\end{equation*}
Moreover, $\tilde\Psi^A$ satisfies the following boundary conditions
\begin{equation*}
\mathcal{B}\tilde\Psi^A=0
\end{equation*}
for all time $t\in[0,T]$. By the uniqueness of solutions of Dirac equations with chiral boundary values, see \autoref{thm:dirac-omega}, we get that $\tilde\Psi=0$, i.e., $\Psi$ is a spinor along the map $\Phi$.
\end{proof}
 To state the short time existence for the Dirac-harmonic map flow, we first recall some basic facts of heat kernels on Riemannian manifolds. An important property is that the heat kernel is almost Euclidean \cite{Chavel1984eigenvalues,Li2012geometric}. In other words, if $p$ is a heat kernel, then $p$ and $\mathcal{E}$ are of the same order, locally uniformly in $(x,y)$ as $t\to 0_+$, and  a similar statement holds for the first derivatives of $p$ and $\mathcal{E}$, where
 \begin{equation*}
 \mathcal{E}(x,y,t)=(4\pi t)^{-m/2}e^{-\dist^2(x,y)/(4t)}.
 \end{equation*}
 One can show that the Dirichlet heat kernel $h(x,y,t)$ is also almost Euclidean \cite{Chavel1984eigenvalues}, hence
 \begin{gather*}
 h(x,y,t)\leq ct^{-m/2}e^{-\dist^2(x,y)/(4t)},\\
 \intertext{and}
 \norm{\nabla h(x,y,t)}\leq ct^{-m/2-1}e^{-\dist^2(x,y)/(4t)}\dist(x,y).
 \end{gather*}
 We summarize these properties in
 \begin{lem}[See \cite{Chavel1984eigenvalues,Jost2013partial}]\label{lem:heat-kernel}For ever $\beta>0$, there exists a constant $c=c(\beta)$ such that
 \begin{gather*}
h(x,y,t)\leq c(\beta) t^{-m/2+\beta}\dist(x,y)^{-2\beta},\\
\norm{\nabla h(x,y,t)}\leq c(\beta)t^{-m/2-1+\beta}\dist(x,y)^{1-2\beta},
\intertext{and}
\norm{\nabla_{\vect{n}}h(x,y,t)}\leq c(\beta)t^{-m/2-1+\beta}\dist(x,y)^{2-2\beta},\quad x,y\in\partial M,
\end{gather*}
as $t\to 0_+$.
 \end{lem}
 \begin{proof}It is a consequence of the following inequality
 \begin{equation*}
 x^{\beta}e^{-x}\leq\beta^{\beta}e^{-\beta},\quad\forall x>0,\beta>0.
 \end{equation*}
 The improvement in the exponent of $\dist(x,y)$ in the third inequality is due to the fact that the derivative is in the direction normal to the boundary $\partial M$.
 \end{proof}
Now we can give a proof of the main \autoref{thm:main}. For the short time existence of the harmonic map heat follow, we refer the reader to \cite{Eells1964harmonic,Hamilton1975harmonic, Li1991heat, Lin2008analysis}.
\begin{proof}[\textbf{Proof of \autoref{thm:main}}] We will split the proof into four steps.
\begin{enumerate}[Step I]
\vspace{3ex}
\item\textbf{Short time existence for the flow \eqref{eq:DHF} and \eqref{eq:BDHF}}.
\vspace{3ex}
\par
Let $h(x,y,t)$ be the Dirichlet heat kernel of $M$. Define an operator $T$ by
\begin{equation*}
\begin{split}
&Tu(x,t)=u_0(x,t)-\int_0^t\int_Mh(x,y,t-\tau)\left(\Omega(u)\cdot\dif u+\hin{\tilde\Omega(u,\Psi(u))}{\dif u}\right)(y,\tau)\diff y\diff\tau,
\end{split}
\end{equation*}
where
\begin{equation*}
u_0(x,t)=\int_Mh(x,y,t)\phi(y,0)\diff y-\int_0^t\int_{\partial M}\dfrac{\partial h}{\partial\vect{n}_y}(x,y,t-\tau)\phi(y,\tau)\diff\sigma(y)\diff\tau.
\end{equation*}
Here
\begin{gather*}
\Omega(u)=[\nu(u),\dif\nu(u)],
\intertext{and}
\tilde\Omega(u,\Psi(u))=\dfrac12R^A_{BCD}(u)\rin{\Psi(u)^C}{e_i\cdot\Psi(u)^D}\eta^i\eqqcolon\mathcal{R}(u)(\Psi(u),\Psi(u)),
\end{gather*}
where $\Psi(u)$ is the unique solution of
\begin{equation*}
\begin{cases}
\dirac\Psi^A=-\Omega(u)^A_B\cdot\Psi^B,&M;\\
\mathcal{B}\Psi^A=\mathcal{B}\psi^A,&\partial M
\end{cases}
\end{equation*}
according to \autoref{thm:dirac-omega}.
\par
It is clear that $u_0$ is the unique solution of
\begin{equation*}
\begin{cases}
\dfrac{\partial}{\partial t}u_0=\Delta u_0,&M\times(0,\infty);\\
u_0=\phi,&\partial M\times[0,\infty)\cup M\times\set{0}.
\end{cases}
\end{equation*}
For every $\varepsilon>0$ and each $u\in\cap_{0<t<\varepsilon}C^{1,0,0}(\bar M\times[t,\varepsilon])\cap C^0(\bar M_{\varepsilon})$, define the norm
\begin{equation*}
\norm{u}\coloneqq\norm{u}_{C^0(\bar M\times[0,\varepsilon])}+\sup_{t\in[0,\varepsilon]}\norm{\nabla u(\cdot,t)}_{C^0(\bar M)}.
\end{equation*}
Let $X^{\varepsilon}_{\phi}$ be the completion of the following subset of $C^0(\bar M_{\varepsilon})$
\begin{equation*}
\set{u\in\cap_{0<t<\varepsilon}C^{1,0,0}(\bar M\times[t,\varepsilon])\cap C^0(\bar M_{\varepsilon}):u=\phi\ \text{on}\ \mathcal{P}M_{\varepsilon}}.
\end{equation*}
Here
\begin{equation*}
M_{\varepsilon}\coloneqq M\times(0,\varepsilon],\quad\mathcal{P}M_{\varepsilon}\coloneqq\partial M\times[0,\varepsilon]\cup M\times\set{0}.
\end{equation*}
For $u\in X_{\phi}^{\varepsilon}$, according to \autoref{thm:dirac-omega} we have that for large $p$
\begin{equation*}
\begin{split}
\norm{\Psi(\cdot,t)}_{C^{\alpha}(\bar M)}\leq&c(p)\norm{\Psi(\cdot,t)}_{W^{1,p}(M)}\leq C(\norm{u})\norm{\mathcal{B}\Psi_0(\cdot,t)}_{W^{1-1/p,p}(\partial M)}\\
\leq&C(\norm{u})\norm{\mathcal{B}\Psi_0(\cdot,t)}_{C^{1,\alpha}(\partial M)}\leq C(\norm{u})\norm{\mathcal{B}\psi}_{C^{1,0,\alpha}(\partial M_{\varepsilon})}.
\end{split}
\end{equation*}
As a consequence,
\begin{equation*}
T:X^{\varepsilon}_{\phi}\To X^{\varepsilon}_{\phi}
\end{equation*}
is well defined. For $\delta>0$, let $B_{\delta}=\set{u\in X^{\varepsilon}_{\phi}:\norm{u-u_0}\leq\delta}$.
\par
According to \autoref{lem:heat-kernel}, for every $\beta>0$,
\begin{gather*}
h(x,y,t)\leq c(\beta) t^{\beta-m/2}\dist(x,y)^{-2\beta},\\
\norm{\nabla h(x,y,t)}\leq c(\beta)t^{\beta-1-m/2}\dist(x,y)^{1-2\beta},
\intertext{and}
\norm{\nabla_{\vect{n}}h(x,y,t)}\leq c(\beta)t^{\beta-1-m/2}\dist(x,y)^{2-2\beta},\quad x,y\in\partial M.
\end{gather*}
\par
\vspace{3ex}
\begin{enumerate}[1)]
\item $u_0\in X^{\varepsilon}_{\phi}$.
\par
\vspace{3ex}
Let $v_0=u_0-\phi$, then
\begin{equation*}
\begin{cases}
\partial_tv_0-\Delta v_0=\Delta\phi-\partial_t\phi\eqqcolon f,&M_T;\\
v_0=0,&\mathcal{P}M_T.
\end{cases}
\end{equation*}
Since $\phi\in C^{2,1,\alpha}(\bar M_T)$, we know that $f\in C^{0,0,\alpha}(\bar M_T)$ and
\begin{equation*}
\norm{f}_{C^{0}(\bar M_T)}\leq c\norm{\phi}_{C^{2,1,0}(\bar M_T)},\quad\norm{f}_{C^{0,0,\alpha}(\bar M_T)}\leq c(\alpha)\norm{\phi}_{C^{2,1,\alpha}(\bar M_T)}.
\end{equation*}
 Moreover, $v_0$ can be given in the following formula, for $(x,t)\in M_T$,
\begin{equation*}
v_0(x,t)=\int_0^t\int_Mh(x,y,t-\tau)f(y,\tau)\diff y\diff\tau.
\end{equation*}
The Schauder estimates imply that $v_0\in\cap_{0<t<T} C^{2,1,\alpha}(\bar M\times[t,T])\cap C^0(\bar M_T)$. The following estimates follow by straightforward computations.
\begin{gather*}
\abs{v_0(x,t)}\leq\norm{\phi}_{C^{2,1,0}(\bar M_{\varepsilon})}\varepsilon,\\
\abs{\nabla v_0(x,t)}\leq c(\beta)\norm{\phi}_{C^{2,1,0}(\bar M_{\varepsilon})}\varepsilon^{\beta-1},
\end{gather*}
for all $(x,t)\in M_{\varepsilon}$ and $\beta\in(m/2,(m+1)/2)$. In fact,
\begin{gather*}
\begin{split}
\abs{v_0(x,t)}\leq&\int_0^t\int_Mh(x,y,t-\tau)\abs{f(y,\tau)}\diff y\diff\tau\\
\leq&\norm{f}_{C^0(\bar M_t)}\int_0^t\int_Mh(x,y,t-\tau)\diff y\diff\tau\leq\norm{f}_{C^0(\bar M_t)}t,
\end{split}
\intertext{and}
\begin{split}
\abs{\nabla v_0(x,t)}\leq&\int_0^t\int_M\abs{\nabla_xh(x,y,t-\tau)}\abs{f(y,\tau)}\diff y\diff\tau\\
\leq&c(\beta)\norm{f}_{C^0(\bar M_t)}\int_0^t\int_M\abs{t-\tau}^{-2+\beta}\dist(x,y)^{1-2\beta}\diff y\diff\tau\leq c(\beta)\norm{f}_{C^0(\bar M_t)}t^{\beta-m/2}.
\end{split}
\end{gather*}
where $\beta\in(m/2,(m+1)/2)$.
Therefore $u_0\in X^{\varepsilon}_{\phi}$, and for $\delta,\varepsilon$ both small we have that $u_0\subset\tilde N$ if $\phi\subset N$.
\par
\vspace{3ex}
\item For $\varepsilon$ small, $T(B_{\delta})\subset B_{\delta}$.
\par
\vspace{3ex}
Since $u\in B_{\delta}$, we have $\norm{u}\leq C_1$.
\begin{gather*}
\norm{\Omega(u)}_{C^0(\bar M\times[0,\varepsilon])}=\norm{[\nu(u),\dif\nu(u)]}_{C^0(\bar M\times[0,\varepsilon])}\leq c\sup_{t\in[0,\varepsilon]}\norm{\nabla u(\cdot,t)}_{C^0(\bar M)}\leq c(C_1),
\intertext{and}
\begin{split}
\norm{\tilde\Omega(u,\Psi(u))}_{C^{0}(\bar M\times[0,\varepsilon])}\leq& c\sup_{t\in[0,\varepsilon]}\norm{\Psi(\cdot,t)}^2_{C^0(\bar M)}\\
\leq& c(C_1)\sup_{t\in[0,\varepsilon]}\norm{\mathcal{B}\psi(\cdot,t)}^2_{H^{1-1/p,p}(\partial M)}\\
\leq& c(C_1)\norm{\mathcal{B}\psi}^2_{C^{1,0,\alpha}(\bar\partial M\times[0,\varepsilon])}\leq c(C_1)C_2,
\end{split}
\end{gather*}
for $p$ large enough, and the second inequality has used \autoref{thm:dirac-omega}. Hence,
\begin{equation*}
\norm{\Omega(u)}_{C^0(\bar M_{\varepsilon})}+\norm{\tilde\Omega(u,\Psi(u))}_{C^0(\bar M_{\varepsilon})}\leq c(C_1,C_2),\forall u\in B_{\delta}.
\end{equation*}
As a consequence,
\begin{gather*}
\norm{Tu-u_0}\leq c(\beta)c(C_1,C_2)\norm{\dif u}_{C^0(\bar M_{\varepsilon})}\varepsilon^{\beta-m/2},
\end{gather*}
for $\beta\in(m/2,(m+1)/2)$.
\par
\vspace{3ex}
\item $\norm{Tu-Tv}\leq\dfrac12\norm{u-v}$ for $u,v\in B_{\delta}$ and $\varepsilon$ small.
\par
\vspace{3ex}
\par
First, we have,
\begin{equation*}
\begin{split}
Tu(x,t)-Tv(x,t)=&-\int_0^t\int_Mh(x,y,t-\tau)\left(\Omega(u)\cdot\dif u+\hin{\tilde\Omega(u,\Psi(u))}{\dif u}\right)(y,\tau)\diff y\diff\tau\\
&+\int_0^t\int_Mh(x,y,t-\tau)\left(\Omega(v)\cdot\dif v+\hin{\tilde\Omega(v,\Psi(v))}{\dif v}\right)(y,\tau)\diff y\diff\tau.
\end{split}
\end{equation*}
Moreover, notice that
\begin{gather*}
\begin{split}
\Omega(u)-\Omega(v)=&[\nu(u),\dif\nu(u)]-[\nu(v),\dif\nu(v)]\\
=&[\nu(u)-\nu(v),\dif\nu(u)]+[\nu(v),\dif\nu(u)-\dif\nu(v)],
\end{split}
\intertext{and}
\begin{split}
&\tilde\Omega(u,\Psi(u))-\tilde\Omega(v,\Psi(v))\\
=&\mathcal{R}(u)(\Psi(u),\Psi(u))-\mathcal{R}(v)(\Psi(v),\Psi(v))\\
=&\left(\mathcal{R}(u)-\mathcal{R}(v)\right)(\Psi(u),\Psi(u))+\mathcal{R}(v)(\Psi(u)-\Psi(v),\Psi(u))+\mathcal{R}(v)(\Psi(v),\Psi(u)-\Psi(v)),
\end{split}
\end{gather*}
we have
\begin{gather*}
\norm{\Omega(u)-\Omega(v)}_{C^0(\bar M_{\varepsilon})}\leq c(C_1,C_2)\norm{u-v},
\intertext{and}
\begin{split}
\norm{\tilde\Omega(u,\Psi(u))-\tilde\Omega(v,\Psi(v))}_{C^0(\bar M_{\varepsilon})}\leq &c(C_1,C_2)\left(\norm{u-v}+\norm{\Psi(u)-\Psi(v)}_{C^0(\bar M\times[0,\varepsilon])}\right)\\
\leq& c(C_1,C_2)\norm{u-v}.
\end{split}
\end{gather*}
The last inequality follows from the following fact
\begin{equation*}
\begin{cases}
\dirac\left(\Psi(u)^A-\Psi(v)^A\right)=-\Omega(u)^A_B\cdot(\Psi(u)^B-\Psi(v)^B)+\left(\Omega(u)^A_B-\Omega(v)^A_B\right)\cdot\Psi(v)^B,&M;\\
\mathcal{B}(\Psi(u)-\Psi(v))=0,&\partial M.
\end{cases}
\end{equation*}
And \autoref{thm:dirac-omega} implies that for large $p$
\begin{equation*}
\begin{split}
\norm{\Psi(u)-\Psi(v)}_{C^{\alpha}(\bar M)}\leq &c(p)\norm{\Psi(u)-\Psi(v)}_{W^{1,p}(M)}\leq c(p,C_1,C_2)\norm{u-v}\norm{\Psi(v)}_{L^p(E)}\\
\leq& c(p,C_1,C_2)\norm{u-v}.
\end{split}
\end{equation*}
Thus,
\begin{equation*}
\begin{split}
&\norm{\Omega(u)\cdot\dif u+\hin{\tilde\Omega(u,\Psi(u))}{\dif u}-\Omega(v)\cdot\dif v-\hin{\tilde\Omega(v,\Psi(v))}{\dif v}}_{C^{\alpha}(\bar M)}\leq c(p,C_1,C_2)\norm{u-v}.
\end{split}
\end{equation*}
Using a similar argument for $v_0$, one gets that
\begin{equation*}
\norm{Tu-Tv}\leq c(\beta)c(C_1,C_2)\varepsilon^{\beta-m/2}\norm{u-v},
\end{equation*}
where $\beta\in(m/2,(m+1)/2)$.
\end{enumerate}
Therefore, there exists a fixed point of $T$ in $B_{\delta}$, i.e., we have proved the short time existence of \eqref{eq:DHF} and \eqref{eq:BDHF}.
\par
\vspace{3ex}
\item\textbf{Regularity.}
\par
\vspace{3ex}
Let $(\Phi,\Psi)$ be the solution of \eqref{eq:DHF} and \eqref{eq:BDHF}  constructed above. The \autoref{thm:dirac-omega} implies that $\Psi\in L^{\infty}(\bar M_{\varepsilon})$ and hence $\Phi\in\cap_{0<t<\varepsilon} C^{1,0,\alpha}(\bar M\times[t,\varepsilon])\cap C(\bar M_{\varepsilon})$ by using the $L^p$-estimate for the heat equation. For every $0<t,\tau\leq\varepsilon$, we have
\begin{equation*}
\begin{cases}
\begin{split}
\dirac\left(\Psi^A(\cdot,t)-\Psi^A(\cdot,\tau)\right)=&-\Omega^A_B(\cdot,t)\cdot\left(\Psi^B(\cdot,t)-\Psi^B(\cdot,\tau)\right)+\left(\Omega^A_B(\cdot,\tau)-\Omega^A_B(\cdot,t)\right)\cdot\Psi^B(\cdot,\tau),&M;
\end{split}\\
\mathcal{B}\left(\Psi(\cdot,t)-\Psi(\cdot,\tau)\right)=\mathcal{B}\left(\psi(\cdot,t)-\psi(\cdot,\tau)\right), & \partial M.
\end{cases}
\end{equation*}
Again by using \autoref{thm:dirac-omega}, for larger $p$
\begin{equation*}
\norm{\Psi(\cdot,t)-\Psi(\cdot,\tau)}_{C^{\alpha}(\bar M)}\leq c(p,C_1,C_2)\abs{t-\tau}^{\alpha/2}.
\end{equation*}
Thus, $\Psi\in\cap_{0<t<\varepsilon}C^{0,0,\alpha}(\bar M\times[t,\varepsilon])\cap C^0(\bar M_{\varepsilon})$. The Schauder estimate for the heat equation implies that $\Phi\in\cap_{0<t<\varepsilon} C^{2,1,\alpha}(\bar M\times[t,\varepsilon])\cap C^0(\bar M_{\varepsilon})$. The interior Schauder estimate for the Dirac equation implies that $\Psi(\cdot,t)\in C^{2,\alpha}(M)$ for every $t\leq\varepsilon$. If one uses the boundary Schauder estimate, one can get that $\Psi(\cdot,t)\in C^{1,\alpha}(\bar M)$ for $t\leq\varepsilon$.
\par
Suppose that
\begin{equation*}
\limsup_{t<T_1,t\to T_1}\norm{\dif\Phi(\cdot,t)}_{C^0(\bar M)}<\infty,
\end{equation*}
the discussion above implies that this flow can be extended to a larger time $T_1'>T_1$, hence $T_1$ is not the maximum time which is a contradiction.
\par
\vspace{3ex}
\item\textbf{Uniqueness.}
\par
\vspace{3ex}
Finally, we state the uniqueness. Suppose that $(\Phi_i,\Psi_i)$ are solutions of \eqref{eq:DHF} and \eqref{eq:BDHF}. Let $u=\Phi_1-\Phi_2,\eta=\Psi_1-\Psi_2$, then
\begin{gather*}
\abs{\partial_tu-\Delta u}\leq C\abs{\nabla u}+C\abs{u}+C\abs{\eta},
\intertext{and}
\abs{\dirac\eta+\Omega_1\cdot\eta}\leq C\abs{\nabla u}+C\abs{u}.
\end{gather*}
Hence, applying \autoref{thm:dirac-omega} and using the same computation as for $v_0$, we get that
\begin{gather*}
\norm{\eta(\cdot,t)}_{C^0(\bar M)}\leq C\norm{u},\\
\intertext{and}
\norm{u}\leq C\norm{u}\varepsilon^{\beta-m/2}
\end{gather*}
holds for $0<\varepsilon\leq T_1$ and $\beta\in(m/2,(m+1)/2)$, where
$\norm{\cdot}$ is the norm corresponding to $M_{\varepsilon}$. Thus, if $\varepsilon$ is small, then $\norm{u}=0$, i.e., $u=0$ and hence $\eta=0$.
Then we can prove the uniqueness of the Dirac-harmonic heat flow by iteration.
\par
\vspace{3ex}
\item\textbf{Completion  of the proof.}
\par
\vspace{3ex}
We have actually proved that $\Phi\subset\tilde N$ if $\phi\subset N$. Therefore, we can use \autoref{lem:stational}. As a consequence, $\Phi(\cdot,t)\in N$ and $\Psi(\cdot,t)\in\Sigma M\times\Phi(\cdot,t)^{-1}TN$ for all $0\leq t<T_1$ since $\mathcal{B}\psi(\cdot,t)\in\left(\Sigma M\otimes\phi(\cdot,t)^{-1}TN\right)\vert_{\partial M}$ for all $t$. Then applying \autoref{lem:equivalent}, we finally complete the proof of this theorem.
\end{enumerate}
\end{proof}
\par

\vskip0.2cm

\section{Dirac equations along a map between Riemannian disks}\label{sec:dirac-disk}
In this section, we  discuss a Dirac equation along a smooth map $\phi:M=(D,\lambda\abs{\dif z}^2)\To N=(D,\rho\abs{\dif w}^2)$ where $D=\set{\abs{z}<1}$ is the open unit disk on $\Com$.  Let $\Sigma M$ be the spin bundle on $M$. Consider a Dirac bundle $\Sigma M\otimes\phi^{-1}TN$ and split this Dirac bundle as  (see \cite{Lawson1989spin,Yang2009structure})
\begin{equation*}
\begin{split}
&\Sigma M\otimes\phi^{-1}TN\\
=&\left(\Sigma^+M\otimes\phi^{-1}T^{1,0}N\right)\oplus\left(\Sigma^-M\otimes\phi^{-1}T^{1,0}N\right)\oplus\left(\Sigma^+M\otimes\phi^{-1}T^{0,1}N\right)\oplus\left(\Sigma^-M\otimes\phi^{-1}T^{0,1}N\right),
\end{split}
\end{equation*}
where
\begin{equation*}
\Sigma^+M=\set{\psi\in\Sigma M:\partial_{\bar z}\cdot\psi=0},\quad \Sigma^-M=\set{\psi\in\Sigma M:\partial_{z}\cdot\psi=0}.
\end{equation*}
\par
We identify the Clifford multiplication by the orthogonal bases $\partial_z,\partial_{\bar z}$ with the following matrices (see \cite{Lawson1989spin,Chen2013boundary}):
\begin{equation*}
\partial_z=\lambda^{1/2}
\begin{pmatrix}0&0\\
1&0
\end{pmatrix},
\quad
\partial_{\bar z}=\lambda^{1/2}
\begin{pmatrix}0&-1\\
0&0
\end{pmatrix}.
\end{equation*}
And the spinor $\psi\in\Sigma M\otimes\phi^{-1}TN$ can be written as
\begin{equation*}
\psi=
\begin{pmatrix}f^+\\
f^-
\end{pmatrix}\otimes\partial_{\phi}+
\begin{pmatrix}\tilde f^+\\
\tilde f^-
\end{pmatrix}\otimes\partial_{\bar\phi}
\end{equation*}
with
\begin{equation*}
\begin{pmatrix}1\\
0
\end{pmatrix}\in\Sigma^+M,\quad
\begin{pmatrix}0\\
1
\end{pmatrix}\in\Sigma^-M.
\end{equation*}
The connection on the spin bundle $\Sigma M$  is then given by the following operators (see \cite{Lawson1989spin})
\begin{gather*}
\nabla^{\Sigma M}_{\partial_z}=\dfrac{\partial}{\partial z}+\dfrac14\dfrac{\partial\log\lambda}{\partial z},\quad\nabla^{\Sigma M}_{\partial_{\bar z}}=\dfrac{\partial}{\partial\bar z}+\dfrac14\dfrac{\partial\log\lambda}{\partial\bar z}.
\end{gather*}
Therefore, the Dirac operator on $\Sigma M\otimes\phi^{-1}TN$ is of the following form
\begin{equation*}
\begin{split}
\D=&
2\lambda^{-1/2}\begin{pmatrix}
0&-\dfrac{\partial}{\partial z}-\dfrac14\dfrac{\partial\log\lambda}{\partial z}-\dfrac{\partial\log\rho}{\partial\phi}\dfrac{\partial\phi}{\partial z}\\
\dfrac{\partial}{\partial\bar z}+\dfrac14\dfrac{\partial\log\lambda}{\partial\bar z}+\dfrac{\partial\log\rho}{\partial\phi}\dfrac{\partial\phi}{\partial\bar z}&0
\end{pmatrix}\\
&\oplus2\lambda^{-1/2}\begin{pmatrix}
0&-\dfrac{\partial}{\partial z}-\dfrac14\dfrac{\partial\log\lambda}{\partial z}-\dfrac{\partial\log\rho}{\partial\bar\phi}\dfrac{\partial\bar\phi}{\partial z}\\
\dfrac{\partial}{\partial\bar z}+\dfrac14\dfrac{\partial\log\lambda}{\partial\bar z}+\dfrac{\partial\log\rho}{\partial\bar\phi}\dfrac{\partial\bar\phi}{\partial\bar z}&0
\end{pmatrix}.
\end{split}
\end{equation*}
The chiral boundary operator \cite{Chen2013maximum,Chen2013boundary,Hijazi2002eigenvalue} is given by
\begin{equation*}
\mathcal{B}^{\pm}=\dfrac12
\begin{pmatrix}1&\pm z^{-1}\\
\pm z&1
\end{pmatrix}\oplus
\dfrac12
\begin{pmatrix}1&\pm z^{-1}\\
\pm z&1
\end{pmatrix}.
\end{equation*}
\par
Now we consider the following Dirac equation
\begin{equation*}
\begin{cases}
\D\psi=0,&D;\\
\mathcal{B}^{\pm}\psi=\mathcal{B}^{\pm}\psi_0,&\partial D.
\end{cases}
\end{equation*}
As discussed above, $\D\psi=0$ is equivalent to the following systems
\begin{gather*}
f^+_{\bar z}+\left(\dfrac14(\log\lambda)_{\bar z}+(\log\rho)_{\phi}\phi_{\bar z}\right)f^+=0,\quad f^-_z+\left(\dfrac14(\log\lambda)_{z}+(\log\rho)_{\phi}\phi_z\right)f^-=0,
\intertext{and}
\tilde f^+_{\bar z}+\left(\dfrac14(\log\lambda)_{\bar z}+(\log\rho)_{\bar\phi}\bar\phi_{\bar z}\right)\tilde f^+=0,\quad \tilde f^-_z+\left(\dfrac14(\log\lambda)_{z}+(\log\rho)_{\bar\phi}\bar\phi_z\right)\tilde f^-=0,
\end{gather*}
where the spinor $\psi$ has the  form
\begin{equation*}
\psi=
\begin{pmatrix}f^+\\
f^-
\end{pmatrix}\otimes\partial_{\phi}+
\begin{pmatrix}\tilde f^+\\
\tilde f^-
\end{pmatrix}\otimes\partial_{\bar\phi}.
\end{equation*}
\par
Let $g$ be a solution of the  Riemannian-Hilbert problem
\begin{equation*}
\begin{cases}
g_{\bar z}=\dfrac14(\log\lambda)_{\bar z}+(\log\rho)_{\phi}\phi_{\bar z},&D;\\
\RE g=\log\left(\lambda^{1/4}\rho^{1/2}\right),&\partial D.
\end{cases}
\end{equation*}
All   solutions can be given in the following formulae [see \cite{Begehr1994complex}, p.71, Theorem 21.]
\begin{equation*}
\begin{split}
g(z)=&i\IM g(0)+\log\lambda^{1/4}(z)+\dfrac{1}{2\pi i}\int_{\partial D}\dfrac{\log\left(\rho(\phi(\zeta))^{1/2}\right)}{\zeta}\dfrac{\zeta+z}{\zeta-z}\dif\zeta\\
&+\dfrac{1}{4\pi i}\int_{D}\left(\dfrac{(\log\rho)_{\phi}\phi_{\bar \zeta}(\zeta)}{\zeta}\dfrac{\zeta+z}{\zeta-z}+\dfrac{(\log\rho)_{\bar\phi}\bar\phi_{\zeta}(\zeta)}{\bar\zeta}\dfrac{1+z\bar\zeta}{1-z\bar\zeta}\right)\dif\zeta\wedge\dif\bar\zeta,
\end{split}
\end{equation*}
for all $z\in D$. Then
\begin{gather*}
f^+_{\bar z}+g_{\bar z}f^+=0,\quad\overline{f^-}_{\bar z}+\left((\log\rho+\log\lambda^{1/2})_{\bar z}-g_{\bar z}\right)\overline{f^-}=0,
\intertext{and}
\tilde f^+_{\bar z}+\left((\log\rho+\log\lambda^{1/2})_{\bar z}-g_{\bar z}\right)\tilde f^+=0,\quad\overline{\tilde f^-}_{\bar z}+g_{\bar z}\overline{\tilde f^-}=0.
\end{gather*}
Therefore, there exist four holomorphic functions $A^+,A^-,\tilde A^+,\tilde A^-$ such that
\begin{gather*}
f^+=e^{-g}A^+,\quad f^-=\lambda^{-1/2}\rho^{-1}e^{\bar g}\overline{A^-},\quad\tilde f^+=\lambda^{-1/2}\rho^{-1}e^{g}\tilde A^+,\quad\tilde f^-=e^{-\bar g}\overline{\tilde A^-}.
\end{gather*}
\par
The chirality boundary condition $\mathcal{B}^{\pm}\psi=\mathcal{B}^{\pm}\psi_0$ now is equivalent to
\begin{gather*}
A^+\pm z^{-1}\overline{A^-}=\lambda^{1/4}\rho^{1/2}e^{i\IM g}\left(f^+_0\pm z^{-1}f^-_0\right),\quad\tilde A^+\pm z^{-1}\overline{\tilde A^-}=\lambda^{1/4}\rho^{1/2}e^{-i\IM g}\left(\tilde f^+_0\pm z^{-1}\tilde f^-_0\right),
\end{gather*}
for $z\in\partial D$, where
\begin{equation*}
\psi_0=
\begin{pmatrix}f^+_0\\
f^-_0
\end{pmatrix}\otimes\partial_{\phi}+
\begin{pmatrix}\tilde f^+_0\\
\tilde f^-_0
\end{pmatrix}\otimes\partial_{\bar\phi}.
\end{equation*}
\par
Since the index of $z^{-1}$  is $-1$, the solutions $A^+,A^-,\tilde A^+,\tilde A^-$ must be unique according to \autoref{thm:exist_holomorphic} (see \autoref{sec:app}). In particular, $f^+,f^-,\tilde f^+,\tilde f^-$ are independent  of the choice of $g$. In fact,  we can write any other choice of $g$ by $g+ic$ where $c$ is some real number. Then the solutions $A^+,A^-,\tilde A^+,\tilde A^-$ must be replaced by $e^{ic}A^+,e^{-ic}A^-,e^{-ic}\tilde A^+,e^{ic}\tilde A^-$ respectively. As a consequence, $f^+,f^-,\tilde f^+,\tilde f^-$ do not change.
\par
Next, we construct these solutions by using \autoref{thm:exist_holomorphic}. Denote
\begin{gather*}
F(z)\coloneqq\dfrac{1}{2\pi i}\int_{\partial D}\dfrac{\lambda^{1/4}(\zeta)\rho(\phi(\zeta))^{1/2}e^{i\IM g(\zeta)}(f^+_0(\zeta)\pm \zeta^{-1}f^-_0(\zeta))}{\zeta-z}\dif\zeta,\quad z\notin\partial D
\intertext{and}
\tilde F(z)\coloneqq\dfrac{1}{2\pi i}\int_{\partial D}\dfrac{\lambda^{1/4}(\zeta)\rho(\phi(\zeta))^{1/2}e^{-i\IM g(\zeta)}(\tilde f^+_0(\zeta)\pm \zeta^{-1}\tilde f^-_0(\zeta))}{\zeta-z}\dif\zeta,\quad z\notin\partial D.
\end{gather*}
Then
\begin{gather*}
A^+(z)=F(z),\quad A^-(z)=\pm z^{-1}\overline{F(1/\bar z)},\quad\tilde A^+(z)=\tilde F(z),\quad\tilde A^-(z)=\pm z^{-1}\overline{\tilde F(1/\bar z)},\quad z\in D
\end{gather*}
are the solutions.
\par
To summarize the previous discussion and using \autoref{thm:regularity_holomorphic} (see \autoref{sec:app}), we have the following
\begin{theorem}
Suppose that $\phi\in C^{1+\alpha}(\bar D)$ and $\mathcal{B}^{\pm}\psi_0\in C^{1+\alpha}(\partial D)$ for some $\alpha\in(0,1)$, then there exits a unique solution $\psi\in C^{1+\alpha}(\bar D)$ of
\begin{equation*}
\begin{cases}
\D\psi=0,&D;\\
\mathcal{B}^{\pm}\psi=\mathcal{B}^{\pm}\psi_0,&\partial D.
\end{cases}
\end{equation*}
Moreover, there exists a constant $c=c(\alpha)$ such that
\begin{equation*}
\norm{\psi}_{1+\alpha;D}\leq c\norm{\mathcal{B}^{\pm}\psi_0}_{1+\alpha;\partial D}\norm{\phi}_{1+\alpha;D}.
\end{equation*}
\end{theorem}
\begin{rem}When the domain is $D=\set{\abs{z}<1}$, the MIT bag boundary operator is given by
\begin{equation*}
\mathcal{B}^{\pm}_{MIT}=\dfrac12
\begin{pmatrix}1&\mp iz^{-1}\\
\pm iz&1
\end{pmatrix}\oplus
\dfrac12
\begin{pmatrix}1&\mp iz^{-1}\\
\pm iz&1
\end{pmatrix}.
\end{equation*}
The index of $\mp i z^{-1}$ is $-1$ and therefore we can use the \autoref{thm:exist_holomorphic} and \autoref{thm:regularity_holomorphic} (see \autoref{sec:app}). In particular, the above theorem is also true for  MIT bag boundary value conditions.
\end{rem}
\par

\vskip0.2cm

\appendix
\section{The boundary value problem for the  $\bar\partial$-equation}\label{sec:app}
Let $D=\set{z\in\Com:\abs{z}<1}$ be the open unit disk in the complex plain $\Com$. We say that $(A^+,A^-)$ is a holomorphic function pair on $D$ if $A^+,A^-$ are two holomorphic functions on $D$. Consider the following transformation
\begin{gather*}
\tilde A^+(z)\coloneqq  A^+(z),\quad\forall z\in D^+\coloneqq D,
\intertext{and}
 \tilde A^-(z)\coloneqq\overline{A^-(1/\bar z)},\quad\forall z\in D^-\coloneqq\bar\Com\setminus\bar D.
\end{gather*}
Then $\tilde A^+$ and $\tilde A^-$ are holomorphic in $D^+$ and $D^-$ respectively. Moreover, if $A^+,A^-$ satisfy the following boundary condition
\begin{equation*}
A^+-\varphi\overline{A^-}=f,\quad\text{on}\ \partial D,
\end{equation*}
then $\tilde A^+,\tilde A^-$ satisfy the following boundary condition
\begin{equation*}
\tilde A^{+}-\varphi\tilde A^-=f,\quad\text{on}\ \partial D.
\end{equation*}
\begin{theorem}[See \cite{Begehr1994complex}, p.39, Theorem 14., p.42, Theorem 15., p.11, Theorem 5.]\label{thm:exist_holomorphic}Suppose that $\varphi\in C^{\alpha}(\partial D)$ with $\varphi(\zeta)\neq 0$ for all $\zeta\in\partial D$, where $0<\alpha<1$. Let $f\in C^{\alpha}(\partial D)$ and  $\kappa$ be the index of $\varphi$. Then if $\kappa\geq0$, there exist exactly $\kappa+1$ linearly independent holomorphic function pairs $(A^+,A^-)$ on $D$  such that
\begin{equation*}
A^+-\varphi\overline{A^-}=0
\end{equation*}
is satisfied on the boundary $\partial D$. If $\kappa=-1$, then there exists a unique holomorphic function pair $(A^+,A^-)$ such that
\begin{equation*}
A^+-\varphi\overline{A^-}=f
\end{equation*}
holds on the boundary. Moreover, $A^{\pm}\in C^{\alpha}(\bar D)\cap A(D)$. In fact, if we set
\begin{gather*}
\gamma(z)\coloneqq\dfrac{1}{2\pi i}\int_{\partial D}\dfrac{
\log\left(\zeta\varphi(\zeta)\right)}{\zeta-z}\diff\zeta,\quad z\notin\partial D,
\intertext{and}
\psi(z)\coloneqq\dfrac{1}{2\pi i}\int_{\partial D}\dfrac{f(\zeta)e^{-\gamma(\zeta)}}{\zeta-z}\diff\zeta,\quad z\notin\partial D,
\end{gather*}
then the holomorphic pair $(A^+,A^-)$ is
\begin{equation*}
A^+(z)=e^{\gamma(z)}\psi(z),\quad A^-(z)=z^{-1}\overline{e^{\gamma(1/\bar z)}\psi(1/\bar z)},\quad z\in D.
\end{equation*}
Moreover, there exists a constant $c=c(\alpha)$ such that
\begin{equation*}
\norm{A^{\pm}}_{\alpha;D}\leq c\norm{\varphi}_{\alpha;D}\norm{f}_{\alpha;D}.
\end{equation*}
\end{theorem}
We also need the following Schauder estimate.
\begin{theorem}[See \cite{Begehr1994complex}, p.11, Theorem 5., p.84, Theorem 29.]\label{thm:regularity_holomorphic}
Suppose that $f\in C^{\alpha}(\bar D)$ and $h\in C^{1+\alpha}(\partial D)$ for some $0<\alpha<1$, then every solution of
\begin{equation*}
\begin{cases}s
g_{\bar z}=f,&D;\\
\RE g=h,&\partial D
\end{cases}
\end{equation*}
is of class $C^{1+\alpha}(\bar D)$. Moreover, there exists a constant $c=c(\alpha)$ such that
\begin{equation*}
\norm{g}_{1+\alpha;D}\leq c\left(\norm{f}_{\alpha;D}+\norm{h}_{1+\alpha;\partial D}+\abs{\IM g(0)}\right).
\end{equation*}
\end{theorem}


\end{document}